\documentclass[reqno]{amsart}
\usepackage{hyperref}
\usepackage{amssymb}
\usepackage{amsmath}
\usepackage{amsthm}
\usepackage{mathtools}
\usepackage[utf8]{inputenc}
\usepackage[T1]{fontenc}
\usepackage{enumerate}
\usepackage[svgnames]{xcolor}
\usepackage{graphicx}

\hypersetup{
  pdftitle={On a singularly
perturbed Gross-Pitaevskii equation},
  pdfauthor={Isabella Ianni <isabella.ianni@unina2.it>, Stefan Le Coz
    <slecoz@math.univ-toulouse.fr>, Julien Royer <julien.royer@math.univ-toulouse.fr>},
  pdfsubject={On a singularly
perturbed Gross-Pitaevskii equation},
  pdfkeywords={Gross-Pitaevskii equation, nonlinear
  Schr\"odinger equation, singularly perturbed equations, partial
  differential equations on graphs, Cauchy problem, stationary wave, black soliton, orbital stability, instability} ,
  colorlinks=true,
  linkcolor=DarkBlue  ,          
  citecolor=DarkRed,        
  filecolor=DarkMagenta,      
  urlcolor=DarkGreen,           
}

\newtheorem{theorem}{Theorem}[section]
\newtheorem{proposition}[theorem]{Proposition}
\newtheorem{lemma}[theorem]{Lemma}
\newtheorem{corollary}[theorem]{Corollary}

\theoremstyle{definition}
\newtheorem{definition}[theorem]{Definition}

\theoremstyle{remark}
\newtheorem{remark}[theorem]{Remark}
\newtheorem*{notation}{Notation}
\newtheorem*{acknowledgment}{Acknowledgment}

\numberwithin{equation}{section}

\DeclarePairedDelimiter{\norm}{\lVert}{\rVert}
\DeclarePairedDelimiter{\abs}{\lvert}{\rvert}

\newcommand{\dual}[2]{\left\langle #1,#2 \right\rangle}

\newcommand{\eps}{\varepsilon}

\newcommand{\N}{\mathbb{N}}
\newcommand{\R}{\mathbb{R}}
\newcommand{\C}{\mathbb C}

\DeclareMathOperator{\sech}{sech}
\DeclareMathOperator{\sign}{sign}
\DeclareMathOperator{\supp}{supp}

\renewcommand{\leq}{\leqslant}
\renewcommand{\geq}{\geqslant}

\DeclareMathAlphabet{\mathpzc}{OT1}{pzc}{m}{it}
\renewcommand{\Re}{\mathcal R\!\mathpzc{e}}
\renewcommand{\Im}{\mathcal I\!\mathpzc{m}}

\newcommand{\Hg}{H_\gamma}

\newcommand{\Ec}{{\mathcal E}} 
\newcommand{\Sc}{{\mathcal S}}
\newcommand{\st}{\,:\,}	

\renewcommand{\a}{\alpha}
\renewcommand{\b}{\beta}
\newcommand{\g}{\gamma}
\newcommand{\G}{\Gamma}
\renewcommand{\d}{\delta}

\newcommand{\e}{\varepsilon}
\newcommand{\z}{\zeta} 

\renewcommand{\th}{\theta}

\renewcommand{\k}{\kappa}

\renewcommand{\l}{\lambda}
\renewcommand{\L}{\Lambda}

\newcommand{\n}{\nu}

\newcommand{\s}{\sigma}

\renewcommand{\t}{\tau}
\newcommand{\f}{\varphi}
\newcommand{\vf}{\phi}
\newcommand{\h}{\chi}
\newcommand{\p}{\psi}

\newcommand{\nr}[1]{\left\Vert #1\right\Vert}         
\newcommand{\innp}[2]{\left< #1 , #2 \right>}         

\newcommand {\limt}[2]{\xrightarrow[#1 \to #2]{}}
\newcommand{\pppg}[1] {\left< #1 \right>}

\newcommand{\tHg}{\tilde H_\g}
\newcommand{\tSgHg}[1] {T_\g(#1)}
\newcommand{\tSgHo}[1] {T_0(#1)}

\newcommand{\WW}{W_{R}(T)}

\newcommand{\WWWbis}{\tilde W_{\tilde R}(\tilde T)}

\newcommand{\ie}{{\it{i.e. }}}

\begin{document}

\title[On a singularly
perturbed Gross-Pitaevskii equation]
{On the Cauchy problem and the black solitons of a singularly perturbed Gross-Pitaevskii equation}

\author[I.~Ianni]{Isabella Ianni}
\thanks{The work of I. I. is partially supported by INDAM-GNAMPA}

\address[Isabella Ianni]{
  Dipartimento di Matematica e Fisica
  \newline\indent
  Seconda Universit\`a di Napoli
  \newline\indent
  Viale Lincoln 5, 81100 Caserta
  \newline\indent
  Italia}
\email[Isabella Ianni]{isabella.ianni@unina2.it}

\author[S.~Le Coz]{Stefan Le Coz}
\thanks{The work of S. L. C. is 
  partially supported by ANR-11-LABX-0040-CIMI within the
  program ANR-11-IDEX-0002-02 and  ANR-14-CE25-0009-01}

\author[J. Royer]{Julien Royer}
\thanks{The work of J. R. is 
  partially supported by ANR-11-BS01019-01.}

\address[Stefan Le Coz and Julien Royer]{Institut de Math\'ematiques de Toulouse,
  \newline\indent
  Universit\'e Paul Sabatier
  \newline\indent
  118 route de Narbonne, 31062 Toulouse Cedex 9
  \newline\indent
  France}
\email[Stefan Le Coz]{slecoz@math.univ-toulouse.fr}
\email[Julien Royer]{jroyer@math.univ-toulouse.fr}


\date{\today}

\begin{abstract}
  We consider the  one-dimensional Gross-Pitaevskii equation perturbed by a Dirac
  potential. Using a fine
  analysis of the properties of the linear
  propagator,  we study the well-posedness of the
  Cauchy Problem in the energy
  space of functions with modulus 1 at infinity. 
  Then  we show the
  persistence of the stationary black soliton of the
  unperturbed problem as a solution. We also prove the existence of
  another branch of non-trivial stationary
  waves. Depending on the attractive or repulsive nature of the Dirac
  perturbation and of the type of stationary solutions, we
  prove orbital stability via a variational approach, or linear
  instability via a bifurcation argument. 
\end{abstract}

\maketitle

\vspace{-23pt}

\tableofcontents

\section{Introduction}

We consider the  one-dimensional singularly perturbed
Gross-Pitaevskii equation 
\begin{equation}\label{eq:gp}
  iu_t+u_{xx}-\gamma\delta u+(1-|u|^2)u=0,
\end{equation}
with the boundary condition
\begin{equation}
  \label{eq:boundary}
  |u(t,x)|\to 1,\quad \text{as}\quad |x|\to+\infty.
\end{equation}
Here, $u:\R\times\R\to\mathbb C$, $\gamma\in\R$, $\delta$ is the Dirac
distribution at $0$ and the indices denote the derivatives.

The Gross-Pitaevskii equation is a defocusing nonlinear Schr\"odinger
equation with non-standard boundary conditions.
It has numerous applications in physics, in
particular in nonlinear optics or for Bose-Einstein condensates. 
Since we assume that $|u|\to 1$ at infinity,  a 
rich nonlinear dynamics is possible.
In particular, there exist solutions of
\eqref{eq:gp} either stationary or propagating a fixed profile:  the
\emph{dark} and \emph{grey solitons}.

Perturbations of nonlinear Schr\"odinger equations with one or more Dirac
distributions appear in different contexts in Physics and
Mathematics. 

In nonlinear optics, when polarization of light and birefringence are taken
into account in the modeling of optical
fibers, the resulting model is a system of coupled nonlinear
Schr\"odinger equations, see \cite{Ag07}.
In the study of the soliton-soliton collisions (see e.g. \cite{DeLeWe16,IaLe14}),
if one of the component is very narrow, then its effect on the other
via the coupling can be approximated by the Dirac distribution (see
\cite{CaMa95} and the references therein). 
The mathematical phenomena
related to the interaction of a soliton with the Dirac perturbation have been
studied in depth, first in the groundlaying work  by Goodman, Holmes and Weinstein
\cite{GoHoWe04} and then in a series of papers by Datchev, Holmer, Marzuola and/or Zworski \cite{DaHo09,HoMaZw07-1,HoMaZw07-2,HoZw07}.

Dirac distributions also naturally appear for nonlinear Schr\"odinger equations on graphs. The motivation comes from nanotechnology where
networks of quantum wires are modeled by nonlinear Schr\"odinger
equations on graphs with the Laplacian on the edges and  Kirchoff
transmission conditions at the vertices.  
The equation \eqref{eq:gp} constitutes the simplest 
example of a nonlinear Schr\"odinger equation posed on a graph
consisting of only one vertex and two edges. An
introduction to nonlinear Schr\"odinger equations on graphs is provided
by Noja in \cite{No14}. 

The mathematical study of singularly perturbed nonlinear Schr\"odinger equations
started only a few years ago and is currently in very active
development. Several lines of investigation have been followed. One
problem is to understand the effect of the perturbation on the
dispersive nature of the equation. Outstanding progresses have been
made
recently in this direction by Banica and Ignat
\cite{BaIg11,BaIg14}. Another challenge is to analyse the solitons and
their stability.
After the pioneering work of Fukuizumi and
co. \cite{FuJe08,FuOhOz08,LeFuFiKsSi08}, the analysis of solitons for
nonlinear Schr\"odinger equations on graphs has known a tremendous
development under the impulsion of Adami and co.
\cite{AdCaFiNo14-2,AdCaFiNo14,AdNo09,AdNo13,AdNoOr13,AdNoVi13}. Surprising
phenomena appear, e.g. bistability in the recent work of Genoud,
Malomed and Weish\"aupl \cite{GeMaWe16}. Let us mention also the
recent study of the scattering problem by Banica and Visciglia \cite{BaVi16}.

To our knowledge, our work is the first one where the singularly
perturbed Gross-Pitaevskii equation \eqref{eq:gp} with the
non-standard boundary conditions \eqref{eq:boundary} is considered. We are interested by the Cauchy problem for \eqref{eq:gp} and by
the existence and stability of stationary solutions. 
As
detailed below, two main difficulties arise. First,
due to the non-standard boundary conditions, the natural energy space
is not a vector space and we have to rethink entirely the strategy to
solve the Cauchy problem. Second, the presence of the Dirac perturbation
generates subtle modifications on the stationary solutions of the equation, thus its treatment requires a fine analysis, in
particular for the spectral part of the study.

Before presenting our results, we give some preliminaries on the
structure of \eqref{eq:gp}. 
At least formally, we have the conservation of  the energy $E_\gamma$ defined by
\[
E_\gamma(u)=\frac12\norm{u'}_{L^2}^2+\frac\gamma2|u(0)|^2+\frac14\int_\R(1-|u|^2)^2 \, dx,
\]
where $'$ denotes the derivative with respect to the variable $x$.
Then the equation \eqref{eq:gp} can be rewritten into the Hamiltonian form 
\[
iu_t=\partial E_\gamma(u).
\]
This energy is defined in the energy space 
\[
\mathcal E :=\left\{v\in H^1_{\mathrm{loc}}(\R):\; v'\in L^2(\R),\,(1-|v|^2)\in L^2(\R) \right\}.
\]
Unfortunately $\mathcal E$ is not a vector space and this yields several difficulties in the analysis. 
We   will
endow $\mathcal E$ with the structure of a complete metric
space. 
Several choices are possible for the distance (see e.g. the
discussion in \cite{Ge06}). 
In this work, we have used the two distances $d_0$ and $d_\infty$
defined as follows.
For $u,v \in \Ec$ we set 
\begin{align}
  d_0 (u,v) &= \nr{u' -v'}_{L^2} + \abs{u(0) - v(0)} + \nr{\abs u^2 - \abs v^2}_{L^2},\label{eq:d_0}\\
  d_\infty  (u,v) &= \nr{u'-v'}_{L^2} + \nr{u-v}_{L^{\infty}} + \nr{\abs u^2 - \abs v^2}_{L^2}\label{eq:d_infty}.
\end{align}
The size of $u$ in $\Ec$ will be measured with the quantity
\begin{equation} \label{def-abs-Ec}
  \abs u_\Ec^2 = E_0(u) = \frac 1 2 \nr{u'}^2_{L^2} + \frac14\int_\R(1-|u|^2)^2 \, dx.
\end{equation}
Note that we have used  $E_0$ instead of $E_\gamma$ in the definition
of $\abs u_\Ec$ because the Dirac perturbation is not encoded in the
energy space. Notice moreover that $E_\g(u)$ may be negative when $\g$ is negative.

\subsection{The Cauchy Problem}

Our first main result concerns the well-posedness of the Cauchy
Problem for \eqref{eq:gp}  in the energy space $\mathcal E$. 

A lot of research has been devoted in the
last decades to the study of the Cauchy Problem for various dispersive
PDE and one would expect that a classical-looking equation like \eqref{eq:gp}
is already  covered by existing
results. This is however not the case, as most of the works on
dispersive PDE deal with well-posedness in vector function
spaces for localized or periodic functions. 
Because of the condition $|u|\to1$ at infinity, \eqref{eq:gp} does not fall into that category, and the Cauchy Theory for non-vector function spaces like $\Ec$ is still at its
early stages of development (see e.g.  \cite{Ga08,Ge06,Ge08}).
Another difficulty arising when dealing
specifically with \eqref{eq:gp} is the effect of the Dirac
perturbation, which causes a loss
of regularity at $x=0$ for the solution.

Our result is the following. 

\begin{theorem} [The Cauchy Problem] \label{th-pb-cauchy}
  Let $\g \in \R$. Then for any $u_0 \in \Ec$ the problem \eqref{eq:gp} has a
  unique, global, continuous (for $ d_\infty $ and hence $ d_0$) solution $u : \R \to \Ec$ with $u(0) = u_0$.
  Moreover, the following properties are satisfied.

  \begin{itemize}  
  \item[(i)] \emph{Energy conservation:}
    For all $t \in \R$ we have 
    \[
    E_\g(u(t)) = E_\g (u_0).
    \]
  \item[(ii)] \emph{Continuity with respect to the initial condition:}
    For $R > 0$ and $T > 0$ there exists $C > 0$ such that for $u_0,\tilde u_0 \in \Ec$ with $\abs{u_0}_\Ec \leq R$ and $\abs{\tilde u_0}_\Ec  \leq R$ the corresponding solutions $u$ and $\tilde u$ satisfy
    \[
    \forall t \in (-T,T), \quad  d_\infty  \big( u(t) , \tilde u (t) \big) \leq C  d_\infty (u_0,\tilde u_0).
    \]
  \end{itemize}
\end{theorem}

The proof of Theorem \ref{th-pb-cauchy} is based on a fixed point
argument. Several steps are necessary.

A main task is to acquire a good understanding on the linear
propagation. We denote  by
$H_\gamma$
the unbounded self-adjoint operator rigorously defined from the formal
expression $-\partial_{xx}+\gamma\delta$ (see Section \ref{sec:higher}):
\[
\begin{cases}
\Hg = - \displaystyle\frac {d^2}{dx^2},\\
D(\Hg) = \left \{u \in H^1(\R) \cap H^2(\R\setminus\{0\} ) \st u'(0^+) -  u'(0^-) = \g u(0) \right\}.
\end{cases}
\]
Note that $H_\gamma$ differs from the usual second order derivative operator
only by the \emph{jump condition}:
\begin{equation}\label{eq:jump}
  \partial_x u(0^+)-\partial_x u(0^-)=\gamma u(0).
\end{equation}
We start by giving an explicit characterization of the linear group
$e^{-itH_\gamma}$.  Precisely, we decompose the linear group
$e^{-itH_\gamma}$ in a regular part containing  the free propagator
$e^{-itH_0}$ and a singular part $\Gamma(t)$: 
\[
e^{-itH_\gamma}=e^{-itH_0}+\Gamma(t),
\]
and we give an explicit expression for the kernel of 
$\Gamma(t)$.

As expected, the treatment of
the free linear evolution does not cause any trouble, and the
tricky part is to deal with $\Gamma(t)$. In particular, when
$\gamma<0$, the kernel of $\Gamma(t)$ is rather hard to handle, and
we have to find a clever way to decompose it into two parts that are
treatable separately (see the decomposition in Lemma
\ref{lem-decomposition-G}). This decomposition is crucially
involved in the rest of the study of the linear evolution. Whereas the
explicit formula for the kernel was previously derived in the
literature, the decomposition lemma is a new tool to deal with the
propagator $e^{-itH_\gamma}$.

With the explicit formula for the kernel of the propagator and the
decomposition lemma at our disposal, we are equipped for the study of 
the propagator $e^{-it\Hg}$.
We first prove that it defines a
continuous map on $H^1(\R)$. It is clear if $\g = 0$ since
$H_0=-\partial_{xx}$ commutes with derivatives, but it is no longer the case when
$\g \neq 0$. 
Then we extend $e^{-it\Hg}$ to a map 
on the energy space $\Ec$. This map will inherit most of the nice properties of
the unitary group on $L^2(\R)$. 
Finally, we prove that for $u_0 \in \Ec$
the map 
$t \mapsto e^{-it\Hg}  u_0 - u_0$ 
is continuous with values in
$H^1(\R)$.  In other words, it sends functions with non-zero
boundary data at infinity to localized functions. That is a central
point in our analysis.

After the study of the linear propagator, we are ready to tackle the
analysis of the Cauchy Problem.  We rewrite the problem \eqref{eq:gp} in terms of a Duhamel formula, to which we will apply Banach fixed point theorem to prove the local well-posedness.

The last step consists in proving the conservation of the energy. As usual (see
e.g. \cite{Ca03}), we first consider a dense subset of more regular initial data, for which \eqref{eq:gp} has a strong solution. However, the
singular nature of the Dirac perturbation prevents us from working with
functions regular at $0$. Thus, we have constructed the space $X_{\gamma}^2$ of
functions which have locally the $H^2(\R)$ regularity and satisfy at 0 the
jump condition \eqref{eq:jump} generated by the Dirac perturbation.
We prove the conservation of the energy for such an initial condition
$u_0$ and then, by density, for any $u_0 \in \Ec$. Global existence is
then a consequence of energy conservation.

\subsection{The Black Solitons}

When $\gamma=0$, \eqref{eq:gp} admits traveling waves, \ie solutions
of the form $\kappa_c(x-ct)$. 
In this case, a traveling wave of finite energy is either a constant
of modulus $1$ or, for $|c|<\sqrt{2}$ and up to
phase shifts or translations, it has a non-trivial profile given by an explicit
formula.

The nontrivial traveling waves have been the subject of a thorough
investigation in the recent years. When $c\neq0$, they are often
called \emph{grey solitons}, a terminology which stems from nonlinear
optics (such solitons appear grey in the experiments). For $c\neq0$, orbital stability was proved  via the
Grillakis-Shatah-Strauss Theory \cite{GrShSt87,GrShSt90} by Lin
\cite{Li02} and later revisited by Bethuel, Gravejat and Saut
\cite{BeGrSa08} via  the variational method introduced by Cazenave and
Lions \cite{CaLi82}. When $c=0$, the traveling wave becomes a
stationary wave and is now called a \emph{black soliton}. The study of
orbital stability is much trickier when $c= 0$, due to the
fact that the solution vanishes, and it is no longer possible to make
use of 
the so-called hydrodynamical formulation of the Gross-Pitaevskii
equation (see e.g. \cite{BeGrSa08} for details). Nevertheless orbital
stability of the black soliton was proved
via variational methods by Bethuel, Gravejat, Saut and Smets \cite{BeGrSaSm08} and via the inverse scattering transform by G\'erard and Zhang \cite{GeZh09} (see also \cite{DiGa07} 
for an earlier result and numerical simulations). Recently, Bethuel, Gravejat and Smets proved the orbital stability of a chain of solitons of the Gross-Pitaevskii equation  \cite{BeGrSm12-1} as well as asymptotic stability of the grey solitons \cite{BeGrSm12-2} and of the black soliton \cite{GrSm14}. Existence and stability of traveling waves with a non-zero background for equations of type \eqref{eq:gp} with a general nonlinearity was also studied by Chiron \cite{Ch12,Ch13}.

When $\gamma\neq0$, the Dirac perturbation breaks the translation
invariance and traveling waves do not exist anymore. However,
stationary solutions $u(t,x)\equiv u(x)$ solving the ordinary differential equation
\begin{equation}
  \label{SGPdelta}
  u''-\gamma\delta u+ (1-|u|^2)u=0 
\end{equation}
are still expected. In fact, the black soliton $\kappa_0$ is still a
solution to \eqref{eq:gp} when $\gamma\neq0$, and 
other branches of nontrivial solutions bifurcate from the constants of modulus $1$. 
Precisely (see Proposition \ref{prop:existence}) 
the set of  finite-energy solutions to \eqref{SGPdelta} is
\begin{align*}
  \left\{e^{i\theta}\kappa,e^{i\theta}b_\gamma\ :\  \theta\in\R\right\}&\text{ if }\gamma>0,
\\ 
\left\{e^{i\theta}\kappa,e^{i\theta}b_\gamma,e^{i\theta}\tilde
    b_\gamma\ :\  \theta\in\R\right\}&
\text{ if }\gamma<0,
\end{align*}
where (see also Figure \ref{fig:b_g})
\[
\kappa(x):=\tanh \left( \frac{x}{\sqrt{2}}\right),\quad
b_{\gamma}(x):=\tanh \left(
  \frac{|x|-c_{\gamma}}{\sqrt{2}}\right),\quad
\tilde b_{\gamma}(x):=\coth \left( \frac{|x|+c_{\gamma}}{\sqrt{2}}\right),
\]
for $c_{\gamma}:=\frac{1}{\sqrt{2}}\sinh^{-1}\left(
  -\frac{2\sqrt{2}}{\gamma} \right)$.
This existence result is obtained
using ordinary differential equations techniques. The analysis of
\eqref{SGPdelta} is classical when $\gamma=0$ and the difficulty when $\gamma\neq0$
is to deal with the jump condition \eqref{eq:jump}
induced by the Dirac perturbation.

\begin{figure}[htpb!]
\includegraphics[width = 0.45
\linewidth]{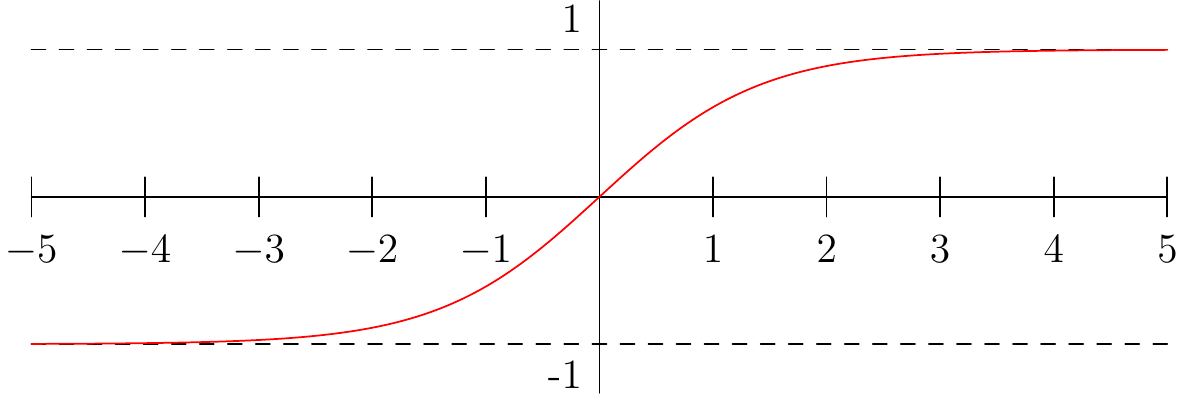}
\includegraphics[width = 0.45 \linewidth]{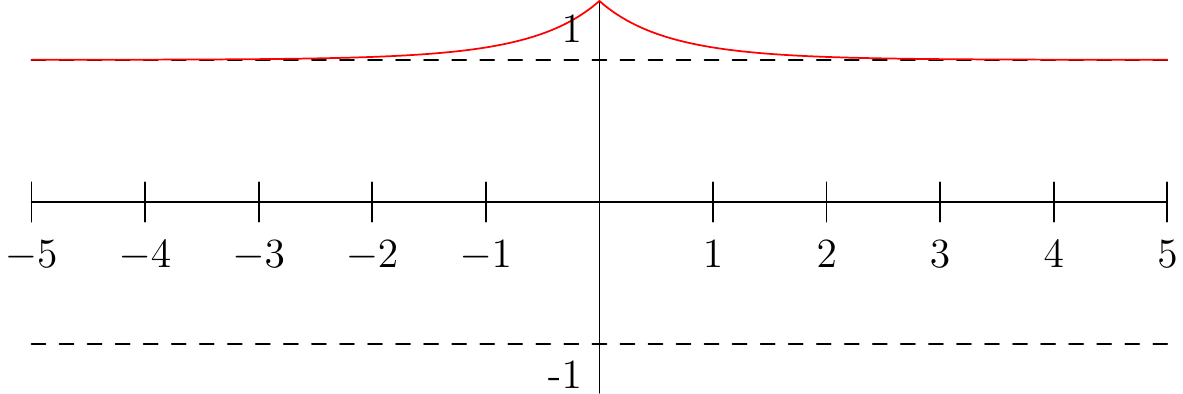}
\\
\includegraphics[width = 0.45 \linewidth]{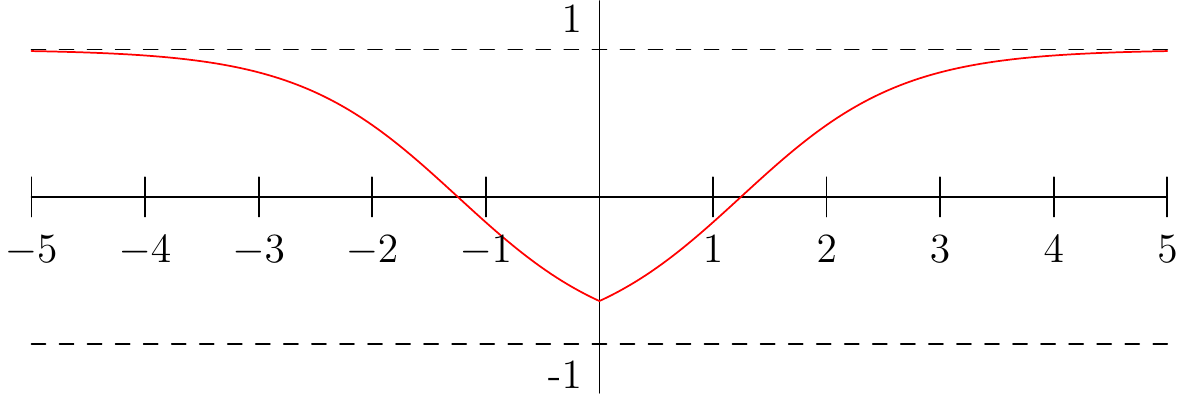}
\includegraphics[width = 0.45 \linewidth]{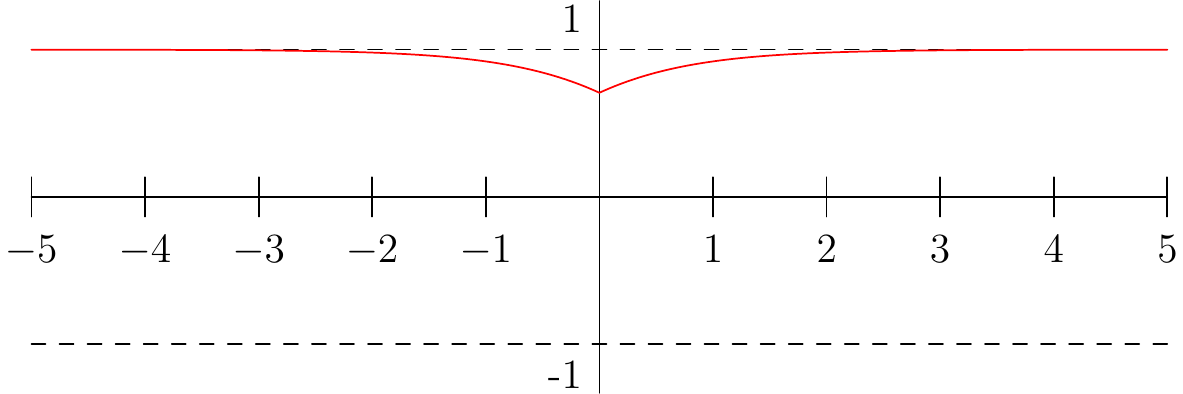}
\caption{Top {Left}: The stationary state $\kappa=\kappa_0=\tanh
  \left(\frac{x}{\sqrt{2}}\right)$. \newline
Top Right: The stationary state $\tilde b_{\gamma}(x):=\coth \left(
  \frac{|x|+c_{\gamma}}{\sqrt{2}}\right),$ $\gamma=-1$. \newline
Bottom: The stationary state $b_\gamma$ for $\gamma=-1$ and $\gamma=1$.}
\label{fig:b_g}
\end{figure}

The next step in the study of stationary solutions to \eqref{eq:gp}
is to understand their stability. To this aim, and when possible, we
give a variational characterization of the stationary solutions. When $\gamma=0$, the traveling waves can be characterized
as minimizers of the energy on a fixed momentum constraint. This is a
non-trivial result due to difficulties in the definition of the
momentum (see \cite{BeGrSa08,BeGrSaSm08}). We will show in our next
result that, depending on the sign of $\gamma$, either
$\tilde b_\gamma$ or
$b_{\gamma}$ can be characterized as minimizers. 
The
minimization problem turns out to be simpler than when
$\gamma=0$, 
and the stationary solutions are in fact global minimizers of the
energy without constraint.

The minimization result is the following.

\begin{proposition}[Variational Characterization]\label{prop:minimization}
  Let $\gamma\in\R\setminus\{ 0\}$. Then we have 
  \[
  m_{\gamma}:=\inf \{E_{\gamma}(v): v \in \Ec\} > - \infty.
  \] 
  Moreover the infimum is achieved at solutions to \eqref{SGPdelta}. Precisely, define 
  \[
  \mathcal G_\gamma:=\{ v\in\mathcal E,\quad E_\g(v)=m_\gamma \}.
  \]
  Then the following assertions hold.
  \begin{itemize}
  \item[(i)]If $\gamma >0$, then $\mathcal G_\gamma=\{ e^{i\theta}b_\gamma, \theta\in\R \}$,
  \item[(ii)] If $\gamma <0$, then $\mathcal G_\gamma=\{
    e^{i\theta}\tilde b_\gamma, \theta\in\R \}$.
  \end{itemize}
  In addition, any minimizing sequence $(u_n)\subset \mathcal E$ such that $E_\gamma(u_n)\to m_\gamma$ verifies, up to a subsequence,
  \[
  d_0(u_n,\mathcal G_\gamma)\to 0.
  \]
\end{proposition}

In the cases covered by Proposition \ref{prop:minimization}, stability will be
a corollary of the variational characterization of the
stationary states as global minimizers of the energy. Let us recall
that  we say that the set $\mathcal
G_\gamma\subset \mathcal E$ is \emph{stable} if for any $\eps>0$ there exists $\delta>0$ such that  for any $u_0\in\mathcal E$ with 
\[
d_0(u_0,\mathcal G_\gamma)\leq \delta,
\]
the solution  $u$ of \eqref{eq:gp} with $u(0)=u_0$ is global and verifies
\[
\sup_{t\in\R}d_0(u(t),\mathcal G_\gamma)\leq \eps.
\]

When the stationary solutions are not minimizers of the energy, we
expect them to be all unstable. In this paper, we treat the case
$\gamma>0$ and we show that $\kappa$ is linearly unstable. 

Linear
instability means that the operator
arising in the linearization of \eqref{eq:gp} around $\kappa$ (see
e.g. \cite{ChGuNaTs07}) admits an eigenvalue with negative real part.
Precisely, consider the linearization of \eqref{eq:gp}
around the kink stationary solution $\kappa(x)$. 
For $u$ solution of \eqref{eq:gp} we write $u=\kappa+\eta$. 
The perturbation $\eta$ verifies 
\begin{equation} \label{eq-eta}
  \eta_t+L\eta+N(\eta)=0,
\end{equation}
where the linear and nonlinear parts are given by
\begin{equation}\label{eq:def-L}
  \begin{aligned}
    L\eta&=-i(\partial_{xx}\eta-\gamma\delta\eta+(1-\kappa^2)\eta-2\kappa^2\Re(\eta)),
    \\
    N(\eta)&=-i(-2\kappa\Re(\eta)\eta-|\eta|^2(\kappa+\eta)).
  \end{aligned}
\end{equation}
The kink $\kappa$ is said to be \emph{linearly unstable} if $0$ is an
unstable solution of the linear equation
\[
\eta_t+L\eta=0.
\]
This is in particular the case if 
$L$
has an eigenvalue $\lambda$ with $\Re(\lambda)<0$. Indeed, such
eigenvalue generates an exponential growth for the corresponding
solution of the linear problem. 
It is
expected that this linear exponential growth translates into nonlinear
instability, as in the theory of Grillakis, Shatah and
Strauss \cite{GrShSt90}, see \cite{GeOh12} for a rigorous proof in the
case of a nonlinear Schr\"odinger equation.

The stability/instability result is the following.

\begin{theorem}[Stability/Instability]\label{thm:instability}
  The following assertions hold.
  \begin{itemize}
  \item[(i)] \emph{(Stability)}
    Let $\gamma\neq 0$. Then the set $\mathcal G_\gamma$ is stable under the flow of \eqref{eq:gp}.
  \item[(ii)] \emph{(Instability)}
    Let $\gamma>0$. Then the kink $\kappa$ is linearly unstable.
  \end{itemize}
\end{theorem}

\begin{remark}
  In the stability result, a solution starting close to a stationary
  wave will always remain close, \emph{up to a phase parameter which may
    vary in time}.
  This type of stability is usually called \emph{orbital
    stability} (see \cite{CaLi82} for an early result on orbital
  stability and \cite[Chapter 8]{Ca03} for a discussion of the orbital
  nature of stability). 
\end{remark}

\begin{remark}
  It is interesting to note that if we define $m_{\gamma,\mathrm{rad}}$ (resp. $m_{\gamma,\mathrm{odd}}$) to be the minimum of $E_{\gamma}(v)$  with $v\in \Ec$, $v$ even (resp. odd), then for any $\gamma\in\R$ we have
  \begin{align*}
\mathcal G_{\gamma,\mathrm{rad}}&
                                                 =\{v\in\mathcal
                                                 E,v\text{ even},E_\gamma(v)=m_{\gamma,\mathrm{rad}}\}=\{
                                                e^{i\theta}b_\gamma,
                                                \theta\in\R
                                                \},\quad\text{if }\gamma>0\\
\mathcal G_{\gamma,\mathrm{rad}}&
                                                 =\{v\in\mathcal
                                                 E,v\text{
                                                 even},E_\gamma(v)=m_{\gamma,\mathrm{rad}}\}=\{
                                                 e^{i\theta}\tilde b_\gamma, \theta\in\R \},\quad\text{if }\gamma<0\\
  \mathcal G_{\gamma,\mathrm{odd}}&=\{v\in\mathcal E,v\text{ odd},E_\gamma(v)=m_{\gamma,\mathrm{odd}}\}=\{ e^{i\theta}\kappa, \theta\in\R \}.
  \end{align*}
In
  particular, the kink $\kappa$ is always stable with
  respect to odd perturbations.
\end{remark}

\begin{remark}
Perturbations
measured in $d_0$ allow for overall phase changes. It is an
interesting question whether a single minimizer can be stable against a class of more
localized perturbations (without phase change)? In \cite{GeZh09}, orbital
stability of the kink was obtained without phase change, for a class of
polynomially decreasing perturbations. However, their method does not
apply in our setting. 
\end{remark}

As already said, part (i) (Stability) in  Theorem \ref{thm:instability} is a
corollary of Proposition \ref{prop:minimization}.
The proof of part (ii) (Instability) of Theorem \ref{thm:instability} relies on a perturbative
analysis partly inspired by \cite{LeFuFiKsSi08}. We first convert $L$ into a new operator $\mathcal L$ by
separating the real and imaginary parts. The operator $\mathcal L$ is
of the form 
\[
\mathcal{L}=\begin{pmatrix}0 & -L_-^\gamma\\L^\gamma_+&0\end{pmatrix},
\]
where $L_+^\gamma$ and $L_-^\gamma$ are selfadjoint operators whose
spectra are well known when $\gamma=0$. Then we use the continuity of these spectra with respect to $\g$ to obtain information on the general case. For instance, $0$ is a simple and isolated eigenvalue of $L_+^0$. For
$\gamma\neq0$, $|\gamma|\ll1$, this eigenvalue moves on one side or the other of the real line, depending on the sign
of $\gamma$. Then we show that
when $\gamma\neq0$ the kernel is always trivial, which implies that the number of negative eigenvalues is constant for $\gamma \in (-\infty,0)$ and for $\g \in (0,+\infty)$. With this kind of information on the spectrum of $L_\pm^\gamma$, we
can prove that $\mathcal{L}$ has a real negative eigenvalue, which is also an eigenvalue for $L$.

The rest of the paper is divided as follows. In Section
\ref{sec:spaces} we analyse the structure of the functional spaces
involved in the analysis, in particular the energy space\;$\Ec$. Section
\ref{sec:propagator} is devoted to the study of the linear propagator. This provides the necessary tools to prove the
well-posedness of the Cauchy Problem in Section
\ref{sec:cauchy}. In Section \ref{sec:existence} we prove the
existence and variational characterizations of the stationary solutions, and we analyse their
stability/instability in Section \ref{sec:stability}.

\begin{acknowledgment}
  The authors are grateful to Masahito Ohta for pointing them out the
  existence of the family of $\coth$ based bound states. They are also
  grateful to the unknown referee for helpful comments.
\end{acknowledgment}

\begin{notation}
  The space $L^2(\R)$ will be endowed with the real scalar product 
  \[
  \innp{u}{v}_{L^2}=\Re\int_{\R}u\bar vdx.
  \]
  The homogeneous Sobolev space $\dot H^1(\R)$ is defined by
  \[
  \dot H^1(\R) = \{u \in H^1_{\mathrm{loc}}(\R) \st u' \in L^2 \}.
  \]
  Given an operator $L$, we denote by $L^*$ its adjoint. As usual, $\Sc=\Sc(\R)$
  will denote the Schwartz space of  rapidly decreasing functions. We
  denote by
  $ C^{\infty}_0(\R)$  
  the set of $ C^{\infty}$ functions from $\R$ to $\C$  with compact
  support.
  For $x\in\R$, we use the Japanese bracket to denote
  \[
  \pppg x=\sqrt{1+|x|^2}.
  \]
\end{notation}

\section{Functional Spaces}
\label{sec:spaces}

\subsection{The Energy Space}

In this section we give the basic properties of the energy space $\Ec$.
Some of the properties presented are already known (see
e.g. \cite{Ge08}) but we give here statements adapted to our
needs. 

Recall that we have defined two distances $d_0$, $d_\infty$ in
\eqref{eq:d_0}-\eqref{eq:d_infty} and $\abs
\cdot_{\mathcal E}$ in \eqref{def-abs-Ec}. This endow $\mathcal E$ with structures of complete metric spaces.
We will use $ d_\infty $ to measure the continuity of the flow in
Theorem \ref{th-pb-cauchy}, while $ d_0$ will be useful for the
stability result.

It is clear that for $u, v \in \Ec$ we have $d_0(u,v) \leq d_\infty(u,v)$. On the other hand, $d_0$ does not control $d_\infty$. Indeed, if for $n \in \N\setminus\{0\}$ and $x \in \R$ we set 
\[
u_n(x) = 1 \quad \text{and} \quad v_n(x) = e^{i\pi \f_n(x)} \quad \text{with} \quad \f_n(x) = \frac {\abs x}{n+ \abs x},
\]
then we have 
\[
d_0 (u_n , v_n)^2 = \nr{v'_n}_{L^2}^2 = \pi^2 \nr{\f_n'}_{L^2}^2 
\limt n {+ \infty} 0
\]
but for all $n$
\[
d_\infty (u_n,v_n) \geq \nr{u_n-v_n}_{L^{\infty}} = 2.
\]

We start by showing that functions in the energy space are in fact
continuous, bounded and with modulus $1$ at infinity. Moreover the quantities $E_\g(u)$ and $\abs u_\Ec$ are comparable. We will see that the first one is preserved for a solution of \eqref{eq:gp} and the second will give the time of existence for the local well-posedness, so the following result will be crucial to obtain global well-posedness.

\begin{lemma} \label{lem-abs-Ec}
Let $u\in \Ec$. Then $u$ is uniformly continuous, bounded and
\[
\lim_{|x|\to+\infty}|u(x)|=1.
\]
Moreover there exists $C >0$ such that for every $u \in \Ec$ we have 
\begin{gather} 
\nr{u}_{L^{\infty}} \leq C (1+ \abs{u}_\Ec^{2/3}), \label{estim-Linfinity-Ec}
\\
\abs {u}_\Ec^{4/3} - C \leq E_\g (u) \leq C (1 + \abs u_\Ec^2). \label{Egamma-absEc}
\end{gather}
\end{lemma}

\begin{proof} 
  Let $u \in \Ec$. Since $u' \in L^2(\R)$, $u$ is uniformly continuous. Assume by contradiction that there exist $\e > 0$ and a sequence $(x_n)$ such that 
  \[
  \lim_{n\to+\infty}|x_n|=+\infty   \quad \text{and} \quad \big|1-|u(x_n)|^2\big|>\varepsilon. 
  \]
  By uniform continuity there exists $\delta>0$ such that for $n\in \N$ and $x\in[x_n-\delta,x_n+\delta]$ we have 
  \begin{equation*}
    \big|1-|u(x)|^2\big|>\frac\varepsilon 2.
  \end{equation*}
  On the other hand, since $(1-\abs u^2) \in L^2(\R)$, we have
  \[
  \lim_{n\to+\infty}\int_{x_n-\delta}^{x_n+\delta}(1-|u|^2)^2dx=0.
  \]
  This gives a contradiction, and hence $|u(x)| \to 1$ as $\abs x \to \infty$. Since $u$ is continuous, we deduce that it is bounded. Now let $v= 1- \abs u^2$. Then $v$ belongs to $H^1(\R)$ and we have
  \[
  \nr{v}_{L^\infty}^2 \lesssim \nr{v}_{L^2} \nr{v'}_{L^2} \lesssim \abs{u}_\Ec^2 \nr{u}_{L^\infty}. 
  \]
  This gives
  \[
  \nr{u}_{L^\infty}^4 \lesssim 1 + \nr{v}_{L^\infty}^2 \lesssim 1 +  \abs{u}_\Ec^2 \nr{u}_{L^\infty},
  \]
and \eqref{estim-Linfinity-Ec} follows. We easily deduce the second inequality of \eqref{Egamma-absEc}. The first one is clear for $\g \geq 0$. When $\g < 0$ we write for some $\tilde C > 0$
\[
E_\g (u) \geq \abs u_\Ec^2 - \frac {\abs \g} 2 \nr{u}_{L^\infty(\R)}^2 \geq \abs{u}^{4/3}_\Ec \left(\abs u_\Ec^{2/3} - \tilde C \right) - \tilde C .
\]
This concludes the proof.
\end{proof}

\begin{lemma}[Continuity of the Energy] \label{lem-E-continue}
  Let $\g \in \R$. Then the energy $E_\g$ is continuous on $(\Ec, d_0)$, and hence on $(\Ec, d_\infty )$. More precisely, for $R > 0$ the functional $E_\g$ is Lipschitz continuous on $\{ u \in \Ec \st \abs u _\Ec \leq R\}$.
\end{lemma}

\begin{proof}
  For $A,B \in \C$ we have
  \begin{equation} \label{eq-A2-B2}
    \begin{aligned}
      \abs{\abs A ^2 - \abs B ^2} = \abs{\Re \big( (A-B)(\overline A+ \overline B) \big)}  \leq \abs {A-B} \big( 2 \abs A + \abs {A-B} \big).
    \end{aligned}
  \end{equation}
  Thus for $u,v \in \Ec$ we have
  \begin{eqnarray*}
    \lefteqn{\abs{E_\g(u) - E_\g(v)}}\\
 && \lesssim \int_\R \abs*{\abs{u'}^2 - \abs{v'}^2} + \abs*{\abs{u(0)}^2 - \abs{v(0)}^2 } + \int_{\R} \abs*{ \abs{1 - \abs u^2}^2 -  \abs{1 - \abs v^2}^2 } \\
 && \lesssim    d_0 (u,v) \big( 1+\abs u _\Ec +  d_0 (u,v) \big).
  \end{eqnarray*}
Notice that we have used \eqref{estim-Linfinity-Ec} to control $\abs{u(0)}$. The result follows.
\end{proof}

We now look at the perturbation of a function $u$ in $\Ec$ by a function in $H^1(\R)$. This will be used to apply the fixed point theorem in the proof of the local well-posedness.

\begin{lemma} \label{lem-Ec-H1}
  The following assertions hold.
  \begin{enumerate} [(i)]
  \item If $u \in \Ec$ and $w \in H^1(\R)$ then $u+w \in \Ec$. 
  \item There exists $C > 0$ such that for $u \in \Ec$ and $w \in H^1(\R)$ we have 
    \[
    \abs{u+w}_\Ec \leq C \left( 1 + \abs u _\Ec \right) \left( 1 + \nr {w}_{H^1}^2 \right). 
    \]
  \item Let $R > 0$. There exists $C_R > 0$ such that for $u_1,u_2 \in \Ec$ and $w_1,w_2 \in H^1(\R)$ with $\max \big( \abs {u_1}_\Ec , \abs{u_2}_\Ec , \nr{w_1}_{H^1} , \nr{w_2}_{H^1} \big) \leq R$ we have
    \[
    d_\infty (u_1 + w_1, u_2 + w_2) \leq C_R  \left( d_\infty  (u_1,u_2) + \nr{w_1 - w_2}_{H^1} \right).
    \]
  \end{enumerate}
\end{lemma}

\begin{proof}
  Let $u \in \Ec$ and $w \in H^1(\R)$. We have $u+ w \in H^1_{\mathrm{loc}}(\R)$ and $u' + w' \in L^2(\R)$. Since $1 - \abs u \in L^2(\R)$, $u \in L^{\infty}(\R)$ and $w \in L^2 \cap L^{\infty}(\R)$ we also have by \eqref{estim-Linfinity-Ec} and Sobolev embeddings
  \begin{align*}
    \nr{1 - \abs{u+w}^2}_{L^2} 
    & \leq \nr{ 1 - \abs{u}^2}_{L^2} + 2
      \nr{u}_{L^{\infty}} \nr{ w}_{L^2} + \nr{w}_{L^{\infty}} \nr{w}_{L^2} \\
    &    \lesssim  \left( 1 + \abs u _\Ec \right) \left( 1 + \nr {w}_{H^1}^2 \right). 
  \end{align*}
  In particular $1 - \abs{u+w}^2 \in L^2(\R)$, and (i) and (ii) are proved.
  Now we consider $R > 0$ and $u_1,u_2,w_1,w_2$ as in (iii). We have
  \begin{align*}
    \nr{(u_1 + w_1) - (u_2 + w_2)}_{L^{\infty}}
    & \leq \nr{u_1 -
      u_2}_{L^{\infty}} + \nr{w_1 - w_2}_{L^{\infty}} \\
    &    \lesssim  d_\infty  (u_1,u_2) + \nr{w_1 - w_2}_{H^1}.
  \end{align*}
  The same applies for the $L^2(\R)$-norms of the derivatives. For the last term in $ d_\infty $ we write
  \begin{align*}
    \nr{ \abs{u_1 + w_1}^2 - \abs{u_2 + w_2}^2 }_{L^2}
    & \leq \nr{ \abs{u_1 }^2 - \abs{u_2}^2 }_{L^2} + \nr{ \abs{w_1 }^2 - \abs{w_2}^2 }_{L^2}\\
    & \quad + 2 \nr{u_1(\overline{w_1}- \overline{w_2})}_{L^2} + 2 \nr{(u_1-u_2)\overline{w_2}}_{L^2}.
  \end{align*}
  The first term is controlled by $ d_\infty (u_1,u_2)$. For the second we use \eqref{eq-A2-B2}. For the third we use \eqref{estim-Linfinity-Ec} to control $\nr{u_1}_{L^\infty}$. Finally for the last term we use the fact that $\nr{u_1-u_2}_{L^{\infty}} \leq  d_\infty (u_1,u_2)$ (which would not be the case with the distance $d_0$).
\end{proof}

\begin{remark} 
  For $u \in \Ec$ fixed, the map $w \in H^1 \mapsto u + w \in \Ec$ is also continuous for the metric $ d_0$. In other words, the last statement of Lemma \ref{lem-Ec-H1} holds with $ d_\infty $ replaced by $ d_0$ when $u_1 = u_2$.
\end{remark}

In order to study the nonlinearity of \eqref{eq:gp}, we set for $u \in \Ec$ 
\begin{equation*}
  F(u) = \big( 1-\abs u ^2 \big) u.
\end{equation*}

\begin{lemma}[Nonlinear Estimates] \label{lem-continuite-F-Ec-H1}
  The function $F$ maps $\Ec$ into $H^1(\R)$. Moreover, for $R > 0$ there exists $C_R >0$ such that for $u , u_1,u_2 \in \Ec$ and $w,w_1,w_2 \in H^1(\R)$ with 
  \[
  \max \big( \abs{u}_\Ec , \abs{u_1}_\Ec , \abs{u_2}_\Ec ,  \nr{w}_{H^1} , \nr{w_1}_{H^1} , \nr{w_2}_{H^1} \big) \leq R,
  \]
  we have
  \begin{gather*}
    \nr{F(u+w)}_{H^1} \leq C_R,\\
    \nr{F(u_1+w_1) - F(u_2+w_2)}_{H^1} \leq C_R \big(  d_\infty  (u_1,u_2) + \nr{w_1-w_2}_{H^1} \big). 
  \end{gather*}
\end{lemma}

\begin{proof}
  Let $u \in \Ec$. We have $(1-\abs u^2) \in L^2(\R)$ (by definition) and $u
  \in L^{\infty}(\R)$ (by Lemma \ref{lem-abs-Ec}), so $F(u) \in
  L^2(\R)$. In addition, we have
  \begin{equation} \label{eq-Fu-prime}
    F(u)' = \big( 1-  2\abs u^2 \big) u'  -u^2 \overline u'.  
  \end{equation}
  Since $u^2 \in L^{\infty}(\R)$ and $u' \in L^2(\R)$ this proves that $F(u)' \in L^2(\R)$ and gives the first statement.
  By Lemmas \ref{lem-abs-Ec} and \ref{lem-Ec-H1} we have 
  \begin{align*}
    \nr{F(u_1 + w_1) - F(u_2 + w_2)}_{L^2}
    & \leq \nr{\abs{u_2 + w_2}^2 - \abs{u_1 + w_1}^2}_{L^2} \nr{u_1 + w_1}_{L^{\infty}}\\
    & \quad + \nr{ 1 - \abs{u_2 + w_2}^2 }_{L^2} \nr{ (u_1 +w_1) - (u_2+w_2)}_{L^{\infty}}\\
    & \lesssim_R   d_\infty  (u_1 + w_1 , u_2 + w_2) \\
    & \lesssim_R    d_\infty  (u_1,u_2) + \nr{w_1-w_2}_{H^1} .
  \end{align*}
  We proceed similarly for $F(u_1 + w_1)' - F(u_2 + w_2)'$, starting from \eqref{eq-Fu-prime} and using \eqref{eq-A2-B2} to estimate 
  \[
  \nr{\abs{u_1 + w_1}^2 - \abs{u_2+w_2}^2}_{L^{\infty}}.
  \]
  This concludes the proof.
\end{proof}

\subsection{Functions with Higher Regularity}
\label{sec:higher}

Functions on $\Ec$ can only be solutions of \eqref{eq:gp} in a weak sense. For computations it will be useful to have a dense subset of function with higher regularity.

We first give a precise meaning to the expression $-\partial_{xx}u +
\g \d u$ which appears in \eqref{eq:gp}. For $u,v\in H^1(\R)$ we have
formally 
\[
\innp{-\partial_{xx}u + \g \d u}{v} = q_\g(u,v),
\]
where $q_\g$ is the sesquilinear form defined on $H^1(\R)$ by
\[
q_\g(u,v) = \int_{\R}u'\bar v'dx + \g u(0)\bar v(0).
\]
This defines a closed form bounded from below on $H^1(\R)$. 
Then we can check that the corresponding selfadjoint operator on $L^2(\R)$ is given by 
\[
\Hg = - \frac {d^2}{dx^2}
\]
on the domain
\begin{equation} \label{dom-Hg}
D(\Hg) = \left \{u \in H^1(\R) \cap H^2(\R\setminus\{0\} ) \st u'(0^+) -  u'(0^-) = \g u(0) \right\}
\end{equation}
(see Theorem VI.2.1 in \cite{kato}). This means that for $u \in  D(\Hg)$ we define $\Hg u$ as the only $L^2(\R)$ function which satisfies $\innp{\Hg u} {\vf} = -\innp{u}{\vf''}$ for all $\vf \in C_0^{\infty}(\R\setminus\{0\})$. Notice that $\Hg$ can also be defined via the approach of selfadjoint extensions (see e.g. \cite[Theorem I.3.1.1]{AlGeHoHo88}).

We remark that functions in $C_0^\infty(\R)$ do not belong to the domain of $\Hg$ in general. For computations in a weak sense, we will use the following space of test functions:
\[
D_0(\Hg) = \left \{ u \in  D(\Hg) \st \supp(u) \text{ is
    compact} \right \}.
\]

It will be useful to apply the theory of selfadjoint operators to $\Hg$ on the Hilbert space $L^2(\R)$. However, functions in $\Ec$ are not in $L^2(\R)$. Set 
\[
X_\g^2 = \left \{ u \in L^{\infty}(\R) \cap \dot H^1(\R) \cap \dot
  H^2(\R\setminus\{0\}) \st u'(0^+) - u'(0^-) = \g u(0) \right \}.
\]
Functions in $X_\g^2$ have the same local properties as functions in $D(\Hg)$ (regularity and jump condition), but the integrability at infinity has been relaxed to include a dense subset of $\Ec$.

\begin{lemma}[Density Results] \label{lem-densite-X2}
  The following assertions hold.
  \begin{enumerate}[(i)]
  \item $D_0(\Hg)$ is dense in $H^1(\R)$.
  \item $X^2_\g \cap \Ec$ is dense in $\Ec$ for the distance $ d_\infty $, and hence for $ d_0$.
  \end{enumerate}
\end{lemma}

\begin{proof}
  It follows from a regularization argument by convolution with a
  mollifier (see e.g. \cite[Lemma 6]{Ge06}) that $X^2_0$ is dense in $\Ec$. Let $u \in X^2_0$. For $n \in \N\setminus\{0\}$ and $x \in \R$ we set $\z_n (x) = 1 + \frac {\g \abs x} 2 e^{-nx^2}$. We have
  \[
  (\z_n u)'(0^+) - (\z_n u)'(0^-) = u(0) \big(\z_n'(0^+) -  \z_n'(0^-) \big) = \g u(0) = \g (\z_n u)(0).
  \]
  This proves that $\z_n u \in X^2_\g$. On the other hand
  \[
  \nr{\z_n - 1}_{L^{\infty}} \limt n \infty 0,
  \quad
  \nr{\z_n'}_{L^2} \limt n \infty 0
  \quad
  \text{and}\quad 
  \nr{1- \abs {\z_n}^2} _{L^2} \limt n \infty 0
  \]
  so
  \[
  d_\infty  (u,\z_n u)\limt n \infty 0,
  \]
  and the second statement is proved. Since $H^2(\R)$ is dense in $H^1(\R)$, we similarly prove that $D_0(\Hg)$ is dense in $H^1(\R)$.
\end{proof}

Formally, we can apply $\Hg$ to functions in $X_\g^2$. However, to
emphasize the fact that a function $ u \in X^2_\g$ is not necessarily
in $D(\Hg)$, we denote by $\tHg u$ the function $-u''$ (again, this is
the only $L^2(\R)$ function which coincides with $-u''$ on
$(0,+\infty)$ and $(-\infty,0)$, in particular for $\vf \in
C_0^{\infty}(\R\setminus\{0\})$ we have $\big<{\tHg u},{\vf}\big> =
\innp{u}{-\vf''}$). Integrations by parts between functions in
$X^2_\g$ and $D_0(\Hg)$ read as follows: for $u \in X^2_\g$ and $\vf \in  D_0(\Hg)$ we have
\[
\big<{\tHg u},{\vf}\big> = \innp{u}{\Hg \vf}.
\]

\section{The Linear Evolution}
\label{sec:propagator}

In this section, we study the propagator associated to the linear part of \eqref{eq:gp}. We naturally begin in $L^2(\R)$, and then we extend this propagator to functions in the energy space $\Ec$. For the proof of Theorem \ref{th-pb-cauchy} it will also be useful to prove some results for the linear evolution in $H^1(\R)$ and $X^2_\g$.

\subsection{\texorpdfstring{The Linear Evolution in $L^2(\mathbb R)$}{The Linear Evolution in L2}}

\label{subsec:propagator}

In the Hilbert space $L^2(\R)$, the selfadjoint operator $\Hg$ generates a unitary group $t \mapsto e^{-it\Hg} \in \mathcal L (L^2(\R))$. In particular, for $u_0 \in  D(\Hg)$ the function $u : t \mapsto e^{-it\Hg} u_0$ belongs to $C^1(\R,L^2(\R)) \cap C^0(\R, D(\Hg))$ and is the unique solution for the problem 
\[
\begin{cases}
  \partial_t u(t) + i \Hg u(t) = 0, \quad \forall t \in\R,\\
  u(0) = u_0.
\end{cases}
\]
The purpose of this section is to describe more explicitly the operator $e^{-it\Hg}$.

It is known that for $t \neq 0$ the kernel $K_0(t)$ of the free propagator $e^{-it H_0}$ is given by
\[
K_0(t,\z)  = \frac 1 {\sqrt {4 i \pi t}} e^{-\frac {\z^2}{4it}}.
\]
As explained in introduction, our purpose is to give an explicit expression for the kernel of $e^{-it\Hg} - e^{-itH_0}$. For $x,y \in \R$ we set $\G(t,x,y) = 0$ if $\g = 0$,
\begin{align*}
  \G(t,x,y)& = - \frac \g 2 \int_0^{+\infty} e^{-\frac {\g s} 2} K_0 (t,s + \abs x + \abs y) ds && \text{if } \g > 0,
  \\
  \G(t,x,y) &=  -   \frac {\abs \g} 2 \int_0^{+\infty} e^{- \frac {\abs \g s} 2} K_0 (t,s - \abs x - \abs y) ds + \frac {\abs \g} 2 e^{i\frac {\g^2 t}4} e^{-\frac {\abs \g (\abs x + \abs y)}2} &&  \text{if } \g < 0.
\end{align*}
Then we denote by $\G(t)$ the operator on the Schwartz space $\Sc$
whose kernel is $\G(t,x,y)$. We first observe that $\G(t)^* = \G(-t)$ for all $t \in \R$.

\begin{proposition}[Description of the Propagator] \label{prop-Gamma}
  Let $t \in \R$. Then $\G(t)$ extends to a bounded operator on $L^2(\R)$ and we have
  \[
  e^{-it\Hg} = e^{-itH_0} + \G(t).
  \]
\end{proposition}

The kernel of $e^{-it\Hg}$ was derived for $\g > 0$ in
\cite{gaveaus86,schulman86, manoukian89}. A more general perturbation
(with $\d$ and $\d'$ interactions) is considered in
\cite{albeveriobd94} (see also \cite{AdamiSacchetti}). Here we give a
proof for any $\g \in \R$. The case $\g < 0$ requires a particular
attention.

For computations on $\G(t,x,y)$ when $\gamma>0$, we will often use the operators 
\begin{equation} \label{def-op-L}
  \begin{aligned}
    L_x (t,s,x,y) &= - \frac {2it \sign(x)}{(s + \abs x + \abs y)} \frac {\partial}{\partial x}, \\
    L_y (t,s,x,y) &= - \frac {2it \sign(y)}{(s + \abs x + \abs y)}\frac {\partial}{\partial y}, \\
    L_s (t,s,x,y) &= - \frac {2it }{(s + \abs x + \abs y)} \frac{\partial}{\partial s}.
  \end{aligned}
\end{equation}
These three operators leave invariant the function
\[ 
(t,s,x,y) \mapsto K_0 (t,s + \abs x + \abs y).
\]
They will be used in integrations by parts to obtain powers of $t$ and negative powers of $x$ and $y$. Thus we also introduce the formal adjoints of these operators. For instance for $L_x$ we set 
\begin{equation} \label{def-L-star}
  L_x(t,s,x,y)^* \vf : y \mapsto  2it \, \frac {\partial}{\partial x} \left(\frac {\sign(y) \vf(y)}{ s + \abs x + \abs y} \right).
\end{equation}

Things will be quite different for $\g < 0$ since $(s - \abs x - \abs y)$ can vanish even
if $(s,x,y) \neq (0,0,0)$. In this case the following decomposition lemma will be of great use.

\begin{lemma}[Decomposition of the Kernel] \label{lem-decomposition-G}
  Assume that $\g < 0$. For $t > 0$ and $\rho \geq 0$, let $g(t,\rho)$ be
  defined by 
  \[
  g(t,\rho) = e^{- \abs \g \rho \sqrt t} \left(\frac 1 {\sqrt {4i \pi t}}\int_{-2\rho \sqrt t}^{+\infty} e^{-\frac {(v- i\g t)^2}{4it}}\, dv - 1 \right).
  \]
  Then the following assertions hold.
  \begin{enumerate} [(i)]
  \item 
    For $t > 0$, the operator $\G$ can be decomposed in 
    \[
    \G(t)  = \G_1(t)  + \G_2(t)
    \]
    where the operators $\G_1(t)$ and $\G_2(t)$ have kernels
    \begin{align*}
      \G_1(t,x,y) &= - \frac {\abs \g} 2 \int_0^{\frac {\abs x + \abs y}2}
                    e^{- \frac {\abs \g s} 2} K_0 (t,s - \abs x - \abs y) ds,
      \\
      \G_2(t,x,y) &= - \frac {\abs \g} 2 e^{i\frac {\g^2 t}4} e^{-\frac
                    {\abs \g (\abs x + \abs y) } 4} g \left( t , \frac {\abs x + \abs
                    y}{4 \sqrt t}
                    \right).
    \end{align*}
  \item 
    For any $T > 0$ the function $g$ is bounded on $(0,T] \times \R_+$. Moreover we have 
    \[
    g(t,\rho) \rightarrow 0\quad\text{as }t \to 0,\,\rho \to +\infty.
    \]
  \end{enumerate}
\end{lemma}

The interest of the decomposition $\G = \G_1 + \G_2$ is that on the
one hand $\abs{s - \abs x - \abs y} \geq \frac {\abs x + \abs y}2$
when $s \leq \frac {\abs x + \abs y}2$, so it will be possible to deal
with the contributions of $\G_1$ as for $\G(t)$ in the case $\g > 0$,
using operators of the form \eqref{def-op-L} with $(s  +\abs x + \abs
y)$ replaced by $(s - \abs x - \abs y)$. On the other hand, $\G_2$
will have nice properties given by the exponential decay in $x$ and $y$ of its kernel.

\begin{proof} [Proof of Lemma \ref{lem-decomposition-G}]
Let $x,y \in \R$ and $t>0$. 
We have 
\begin{align*}
\int_{\frac {\abs x + \abs y}2}^{+\infty} e^{- \frac {\abs \g s} 2} e^{-\frac {(s- \abs x - \abs y)^2}{4it}}  \, ds 
& =  {e^{- \frac {\abs \g (\abs x + \abs y)} 2}}    \int_{-\frac {\abs x + \abs y}2}^{+\infty} e^{- \frac {\abs \g v}2} e^{-\frac {v^2}{4it}}\, dv\\
& =   e^{i\frac {\g^2 t}4}  e^{-\frac {\abs \g (\abs x + \abs y) } 2}  \int_{-\frac {\abs x + \abs y}2}^{+\infty} e^{-\frac {(v- i\g t)^2}{4it}}\, dv ,
\end{align*}
and (i) follows.
To prove (ii), we argue as follows.
For $R > 0$, $t > 0$ and $\rho \geq 0$ we set
\[
I_R(t,\rho) = \frac 1 {\sqrt {4i \pi t}}\int_{-2 \rho \sqrt t}^{2 R \sqrt t} e^{-\frac {(v- i\g t)^2}{4it}}\, dv .
\]
For $t > 0$ and $z \in \C$ we also  set 
\[
f_t(z) =  \frac {1}{\sqrt {4i\pi t}}e^{- \frac {(z-i\g t)^2}{4it}}.
\]
This defines a holomorphic function on $\C$. Then we have 
\[
g(t,\rho) = e^{-\abs \g \rho \sqrt t} \lim_{R \to +\infty} \big(-I_{1,R}(t)  + I_{2,R}(t,\rho)  + I_{3}(t,\rho) \big) 
\]
where for $R > 0$ we have denoted by $I_{1,R}(t) $, $I_{2,R}(t,\rho) $ and $I_{3}(t,\rho) $ the integrals of $f_t$ along the curves $\g_1 : \th \in \big[0,\frac \pi 4\big] \mapsto 2R \sqrt t e^{i\th}$, $\g_2 : r \in [-\rho , R] \mapsto 2r \sqrt t e^{i\frac \pi 4} = r \sqrt {4it}$ and $\g_3 : \th \in \big[0,\frac \pi 4 \big] \mapsto -2\rho \sqrt t e^{i\th}$, respectively. 
We have
\begin{align*}
\lim _{R \to +\infty} \abs{I_{1,R}(t)}
& = \lim _{R \to +\infty}\abs*{\frac {2iR\sqrt t }{\sqrt{4i\pi t} } \int_0^{\frac \pi 4} e^{i\th} \exp\left(-\frac {(2\sqrt t Re^{i\th} -i\g t)^2}{4it} \right) \, d\th} \\
& \leq \lim _{R \to +\infty}\frac {R} {\sqrt{\pi} } \int_0^{\frac \pi 4} \exp \left(- R^2 \sin(2\th) -  {\abs \g R \sqrt t \cos(\th)}  \right) \, d\th\\
& = 0.
\end{align*}
Using the dominated convergence theorem we obtain
\begin{align*}
\lim _{R \to +\infty}I_{2,R}(t,\rho) 
& = \lim _{R \to +\infty}\frac 1 {\sqrt \pi} \int_{-\rho}^R \exp \left(- \frac {(r\sqrt {4it} -i\g t  )^2}{4it} \right) \, dr\\ 
& =\frac 1 {\sqrt \pi} \int_{-\rho}^{+\infty} \exp \left(-r^2 + r \g \sqrt{it} - \frac {i\g^2 t}4 \right) \, dr \xrightarrow [\substack{t \to 0 \\ \rho \to +\infty}]{} 1.
\end{align*}
And finally
\begin{align*}
\abs{I_3(t,\rho)} 
& =  \abs*{-\frac {i\rho} {\sqrt {i \pi}} \int_0^{\frac \pi 4} e^{i\th} \exp \left( - \frac {(-2\rho\sqrt t  e^{i\th} - i\g t)^2}{4it} \right) \, d\th} \\
& \leq  \frac {\rho} {\sqrt { \pi}} \int_0^{\frac \pi 4} \abs*{\exp \left( i \rho^2 e^{2i\th} - \g \rho \sqrt t  e^{i\th} - \frac {i\g^2 t } 4 \right)} \, d\th\\
& \leq \frac {\rho} {\sqrt { \pi}} \int_0^{\frac \pi 4} \exp \left( - \rho^2 \sin(2\th) + \abs \g \rho \sqrt t \cos(\th)  \right) \, d\th.
\end{align*}
At this point we have proved that $I_3$ and hence $\lim_{R \to +\infty} I_R$ are bounded uniformly in $t \in (0,T]$ and $\rho \in [0,\rho_0]$ for any $T,\rho_0 > 0$. For $\rho \geq 1$ we write
\begin{align*}
\abs{I_3(t,\rho)}
& \leq \frac {\rho} {\sqrt { \pi}} \int_0^{\rho^{-3/2}} \exp \left(\abs \g \rho \sqrt t   \right) \, d\th +  \frac {\rho} {\sqrt { \pi}} \int_{\rho^{-3/2}}^{\frac \pi 4} \exp \left( - \rho^2 \sin(2\th)  + \abs\g  \rho \sqrt t  \right) \, d\th\\
& \leq e^{\abs \g \sqrt t \rho} \left(\frac {1} { \sqrt {\pi \rho}} + \frac {\sqrt \pi \rho e^{-\rho^2 \sin(2 \rho^{-3/2})}} {4} \right),
\end{align*}
and hence 
\[
e^{- \abs \g \rho \sqrt t} I_3(t,\rho) \limt \rho {+\infty} 0, 
\]
uniformly in $t$. This concludes the proof.
\end{proof}

Before giving the proof of Proposition \ref{prop-Gamma}, we state a result about the decay of $\G(t) \vf$ when $\vf \in \Sc$. This will be useful when defining $\Gamma$ in the 
distributional sense (see  \eqref{eq:distributional-sense}).

\begin{lemma}[Decay Estimate]\label{lem:decay-estimate}
  Let $\vf \in \Sc$ and $t \in \R$. Then there
  exists $C >0$ which only depends on $t$ and on some semi-norm of $\vf$ in $\Sc$ and such that 
  \[
  \abs {(\G(t) \vf)(x)} \leq C \pppg x^{-2}.
  \]
\end{lemma}

Some computations are common to the proofs of Lemma \ref{lem:decay-estimate} and Proposition \ref{prop-Gamma}.

\begin{proof} [Proof of Lemma \ref{lem:decay-estimate}]
  Since $\G(-t) = \G(t)^*$, it is enough to prove the result for $t >
  0$. Fix such a $t>0$ and
  let $\vf \in \Sc$. We first observe that $\G(t) \vf \in L^{\infty}(\R)$ with 
  \begin{equation} \label{nr-infty-Gt}
    \nr{\G(t) \vf}_{L^{\infty}} \lesssim \nr{\vf}_{L^1}.
  \end{equation}

  Assume that $\g > 0$ and let $x \in \R\setminus\{0\}$. 
  After an integration by parts in $y$, we obtain
  \begin{align*}
    \big(\G(t) \vf\big) (x)
    & = -\frac\gamma2\int_{\R} \int_0^{+\infty}  e^{-\frac {\g s} 2} K_0 (t,s + \abs x + \abs y)  L_y(t,s,x,y)^* \vf  (y) \, ds \, dy\\
    & \quad - 2it\g \vf(0) \int_0^{+\infty} e^{-\frac {\g s} 2} \frac {K_0(t,s + \abs x)}{s+\abs x} \, ds,
  \end{align*}
  where $L_y(t,s,x)^*$ is similar to $L_x(t,s,x)^*$ defined in \eqref{def-L-star}.
  After a similar integration by parts with respect to $s$ we get
  \begin{equation} \label{decomp-Gt-phi}
    \begin{aligned}
      \big(\G(t) \vf\big) (x)
      & = -\frac\gamma2\int_{\R} \int_0^{+\infty}   K_0 (t,s + \abs x + \abs y) L_s^* \big(e^{-\frac {\g s} 2} L_y^* \vf \big)  (y) \, ds \, dy\\
      & \quad - i\g t \int_{\R} \frac {K_0(t,\abs x + \abs y)}{\abs x + \abs y}L_y^* \vf(y) \, dy\\
      & \quad - 2it\g \vf(0) \int_0^{+\infty}  K_0(t,s + \abs x) L_s^* \left(\frac {e^{-\frac {\g s} 2}}{s+\abs x} \right) \, ds\\
      & \quad + 4\g t^2\vf(0) \frac {K_0(t,|x|)}{|x|^2}.
    \end{aligned}
  \end{equation}
  Using the fact that $L_s^*$ and $L_y^*$ are both of order $t$ and
  $|x|^{-1}$, this proves that
  \[
  \abs{\big(\G(t) \vf\big) (x)} \lesssim \frac {t^{3/2}}{x^2} \big( \abs {\vf(0)} + \nr{\vf}_{W^{1,1}} \big).
  \]
  With \eqref{nr-infty-Gt} to control the small values of $|x|$, this concludes the proof when $\gamma>0$. 

  Now assume that $\g < 0$. We use the decomposition of Lemma
  \ref{lem-decomposition-G}. For $\G_1$ we proceed as above and use that
  $\abs{s - \abs x - \abs y} \geq \frac {\abs x + \abs y}2$ when $s \leq
  \frac {\abs x + \abs y}2$. The additional boundary terms given when
  integrating by parts in $s$ are exponentially small in $|x|$ since they  contain $e^{-|\gamma|(|x|+|y|)/4}$. For $\G_2$ we have by Lemma \ref{lem-decomposition-G}:
  \[
  \abs{\big(\G_2(t) \vf\big) (x)} \lesssim e^{-\frac {\abs \g \abs x}4} \nr{\vf}_{L^{\infty}}.
  \]
  This concludes the proof in the case $\gamma<0$.
\end{proof}

\begin{proof}[Proof of Proposition \ref{prop-Gamma}]
  As above, since $\G(-t) = \G(t)^*$ it is enough to prove the result for $t \geq
  0$. 
  Let $\vf \in C_0^{\infty}(\R\setminus\{0\})$. For $t >0$ we set
  \[
  u(t) = e^{-itH_0} \vf + \G(t) \vf \in L^2 .
  \]

  We first show that $u$ is continuous at $0$ and satisfies the
  equation pointwise. 
 For $\g > 0$, \eqref{decomp-Gt-phi} now gives 
  \[
  \abs{\big(\G(t) \vf\big) (x)} \lesssim \frac {t^{3/2}}{\pppg x^2} \nr{\vf}_{W^{1,1}},
  \]
  from which we infer that
  \begin{equation} \label{lim-Gt-L2}
    \lim_{ t\to 0}\nr{\G(t) \vf}_{L^2}= 0.
  \end{equation}
  Let us prove that \eqref{lim-Gt-L2} also holds if $\g < 0$. We
  proceed similarly as when $\gamma>0$ for the contribution of $\G_1$. For the contribution
  of $\G_2$ we consider $t_0 > 0$ such that $\abs y \geq 4 t_0 ^{1/4}$
  for all $y \in \supp \vf$. Then for $t \in (0,t_0]$, $x \in \R$ and $y
  \in \supp(\vf)$ we have $\frac {\abs x + \abs y}{4 \sqrt t} \geq
  t^{-1/4}$, so thanks to Lemma \ref{lem-decomposition-G} we have
  $\G_2(t)\vf \to 0$ in $L^2(\R)$ as $t\to 0$.
  In particular we have $u(t) \to \vf$ as $t \to 0$ in $L^2(\R)$.

The map $u_0 : (t,x) \mapsto (e^{-itH_0} \vf) (x)$ is smooth on $(0,+\infty) \times \R$ and satisfies 
  \[
  (i \partial_t + \partial_{xx}) u_0 = 0.
  \]
  Using differentiation under the integral sign and straightforward computations, we can check that the map $u_\G : (t,x) \mapsto (\G(t) \vf)(x)$ is smooth on $(0,+\infty) \times \R\setminus\{0\}$ with
  \[
  (i \partial_t + \partial_{xx}) u_\G = 0 \quad \text{on } (0,+\infty) \times \R\setminus\{0\}.
  \]
  This implies that the same holds for $u$.

  We claim that for all $t \in
  (0,+\infty)$ the following jump condition is verified
  \begin{equation}\label{eq:jump-in-the-proof}
    \partial_x u(t,0^+)-\partial_x u(t,0^-)=\gamma u(t,0).
  \end{equation}
  Let us make the computations to prove \eqref{eq:jump-in-the-proof}  in the case $\gamma>0$, the
  case $\gamma<0$ being similar. We first remark that for the
  unperturbed part we have
  \begin{equation}
    \label{eq:unperturbed-jump}
    (\partial_x e^{-itH_0}\vf) (0^+) - (\partial_xe^{-itH_0}\vf) (0^-) =0.
  \end{equation}
  For the singular part, we have
  \begin{multline*}
    \left(\partial_x \G(t) \vf \right)(0^\pm) =\int_{\R}\partial_x
    \Gamma(t,0^\pm,y)\vf(y) dy\\
    =\mp\frac{\gamma}{2}\int_{\R}
    \int_0^{+\infty}e^{-\frac{\gamma s}{2}}\frac{s+|y|}{2 i
      t}K_0(t,s+|y|)\vf(y)dsdy.
  \end{multline*}
  Therefore, 
  \begin{equation*}
    (\partial_x \G(t) \vf )(0^+) - (\partial_x \G(t) \vf) (0^-) =-\gamma\int_{\R}
    \int_0^{+\infty}e^{-\frac{\gamma s}{2}}\frac{s+|y|}{2 i
      t}K_0(t,s+|y|)\vf(y)dsdy.
  \end{equation*}
  We recognize that $\partial_s K_0(t,s+|y|)=\frac{s+|y|}{2 i
    t}K_0(t,s+|y|)$, so after an integration by parts in $s$ we obtain
  \begin{eqnarray*}
\lefteqn{    (\partial_x \G(t) \vf) (0^+) - (\partial_x \G(t) \vf) (0^-)} \\
&& = \gamma\int_{\R} K_0(t,|y|)\vf(y)dy -\gamma\frac\gamma2\int_{\R} \int_0^{+\infty}e^{-\frac{\gamma s}{2}}K_0(t,s+|y|)\vf(y)dsdy\\
&& = \gamma (e^{-itH_0}\vf  )(0)+\gamma (\G(t) \vf) (0)=\gamma u(t,0).
\end{eqnarray*}
Combined with \eqref{eq:unperturbed-jump}, this proves \eqref{eq:jump-in-the-proof}.

  We now identify $e^{-itH_\gamma}\vf$ with $u(t)$. 
  Let $\p_0 \in C_0^{\infty}(\R\setminus\{0\}) \subset  D(\Hg)$ and $\p : (t,x) \mapsto (e^{-it\Hg}\p_0)(x)$. For $t \neq 0$ we have
  \begin{eqnarray*}
    \lefteqn{\frac d {dt}  \innp{u(t)}{\p(t)}_{L^2} = \innp{u_t(t)}{\p(t)} -  \innp{u(t)}{i\Hg \p(t)}} \\
 && =  \innp{iu_{xx}(t)}{\p(t)} + \innp{u(t)}{i\p_{xx}(t)} \\
 && = \Re\left( \left(- i u_x(t,0^+) + i u_x(t,0^-) \right)\overline {\p(t,0)}  +  i u(t,0) \left(\overline{\p_x(t,0^+)} - \overline{\p_x(t,0^-)} \right)\right)\\
 && =  \Re\left(- i\g u(t,0) \overline{\p(t,0)} + i\g u(t,0) \overline{\p(t,0)}\right)\\
 && = 0.
  \end{eqnarray*}
  Since the map $t \mapsto \innp{u(t)}{\p(t)}_{L^2}$ is continuous at $t = 0$ this proves that for all $t \in \R$ we have
  \begin{align*}
    \innp{e^{it\Hg}u(t)}{ \p_0} = \innp{u(t)}{\p(t)} =\innp{\vf}{\p_0}   .
  \end{align*}
  Since $C_0^{\infty}(\R\setminus\{0\})$ is dense in $L^2(\R)$ we obtain
  \[
  e^{it\Hg}u(t) = \vf,
  \]
  and hence 
  \[
  u(t) = e^{-it\Hg} \vf.
  \]
  Since $e^{-it\Hg}$ is continuous on $L^2(\R)$, this concludes the proof.
\end{proof}

\begin{remark}
  With Proposition \ref{prop-Gamma} and Lemma \ref{lem:decay-estimate} we obtain that for $\vf \in \Sc$ and $t \in \R$ we have 
  \[
  \abs{e^{-it\Hg} \vf (x)} \lesssim \pppg x^{-2}.
  \]
\end{remark}

\subsection{\texorpdfstring
 {The Linear Evolution in $H^1(\mathbb R)$}
 {The Linear Evolution in H1}
}

Having identified the propagator $e^{-itH_\gamma}$ on $L^2(\R)$, we
now describe its action on $H^1(\R)$. The situation here is quite
different from the case $\gamma=0$, where it follows from the
semi-group theory that $e^{-itH_0}$ defines an isometry on
$H^1(\R)$. We nevertheless can prove the following result:

\begin{proposition}[Action of the Propagator on
  $H^1(\R)$]  \label{prop-lin-evol-H1}
  The following assertions hold.
  \begin{enumerate}[(i)]
  \item Let $w \in H^1(\R)$. Then $e^{-it\Hg} w \in H^1(\R)$ for all $t \in \R$ and the map $t \mapsto e^{-it\Hg} w$ is continuous on $\R$. 
  \item Let $T >0$. Then there exists $C >0$ such that for all $w \in H^1(\R)$ and $t \in [-T,T]$ we have 
    \[
    \nr{e^{-it\Hg} w}_{H^1} \leq C \nr{w}_{H^1}.
    \]
  \end{enumerate}
\end{proposition}

Proposition \ref{prop-lin-evol-H1} is a direct consequence of the
description of the propagator $e^{-itH_\gamma}=e^{itH_0}+\Gamma(t)$
given in Proposition \ref{prop-Gamma}, the fact that the result is already known if $\g = 0$ and the following
result.

\begin{lemma}[Action of $\Gamma$ on $H^1(\R)$] \label{lem-Gamma-H1}
  Let $T > 0$. There exists $C > 0$ such that for $t \in [-T,T]$ and $w \in  H^1(\R)$ we have $\G(t) w \in H^1(\R)$ and 
  \[
  \nr{\G(t) w}_{H^1} \leq C \nr {w}_{H^1}.
  \]
  Moreover the map $t \mapsto \G(t) w$ is continuous from $\R$ to $H^1(\R)$.
\end{lemma}

\begin{proof}
  Let $\vf  \in C_0^{\infty}(\R)$. By Proposition
  \ref{prop-Gamma} we know that the map $t \mapsto \G(t) \vf =
  e^{-it\Hg}\vf - e^{-itH_0} \vf$ is continuous from $\R$ to
  $L^2(\R)$ with $\nr{\G(t) \vf}_{L^2} \leq 2 \nr{\vf}_{L^2}$. Let $t \in \R$ and $x \in \R\setminus\{0\}$. Since $(x,y)
  \mapsto K(t,x,y)$ can be seen as function of $\abs x + \abs y$ we have 
  \begin{align}
    (\G(t) \vf )'(x)
    & = \int_\R \partial_x \G(t,x,y) \vf (y)\, dy = \sign(x) \int_\R \partial_y \G(t,x,y) \sign(y) \vf (y)\, dy \nonumber\\
    & = - \sign(x) (\G(t) (\sign(y) \vf ')) (x) - 2 \sign(x) \G(t,x,0) \vf (0).\label{eq:second-term}
  \end{align}
  By continuity of $\G(t)$ in $L^2(\R)$ we obtain that the first term defines a function in $L^2(\R)$ of size not greater than $2\nr{\vf '}_{L^2}$ and is continuous with respect to $t$. 

  The rest of the proof is devoted to the treatment of  the second term
  in \eqref{eq:second-term}. Since $\abs {\vf (0)} \lesssim \nr{\vf
  }_{H^1}$, it is enough to prove that $t \mapsto \G(t,\cdot,0)$ is
  continuous from $[0,+\infty)$ to $L^2(\R)$ (the continuity on $\R$ will follow by duality).

  First assume that $\g > 0$. 
  For $t > 0$, $x \in \R \setminus \{0\}$ and $\b \geq 0$ we get after an integration by parts with $L_s(t,s,x,0)$
  \begin{multline*}
    \frac \g 2 \int_{\b}^{+\infty} e^{-\frac {\g s} 2} K_0(t,s+\abs x) \, ds\\
    =  \frac {i \g t}{\abs x + \b} e^{-\frac {\g \beta} 2} K_0(t,\b + \abs x) + i\g t \int_\b^{+\infty} K_0(t,s+\abs x) \partial_s \left( \frac {e^{-\frac {\g s}2}} {s + \abs x} \right) \,ds 
  \end{multline*}
  and hence 
  \begin{equation} \label{decomp-G-tx0}
    \abs{\G(t,x,0)} \lesssim \frac \b {\sqrt t} + \frac {\sqrt t}{\abs x + \b}.
  \end{equation}
  Applied with $\beta=1$ for $|x|< 1$ and  $\b = 0$ for $|x|\geq 1$, this proves that $\G(t,\cdot,0)$ belongs to $L^2(\R)$  for $t > 0$ fixed. Moreover the map $t \mapsto
  \G(t,x,0)$ is continuous on $(0,+\infty)$ for all $x \in\R$ and
  \eqref{decomp-G-tx0} is uniform for $t$ in a compact subset of 
  $(0,+\infty)$. By the dominated convergence theorem we obtain that $t \mapsto \G(t,\cdot,0)$ is continuous from $(0,+\infty)$ to $L^2(\R)$. It remains to prove that $\nr{\G(t,\cdot,0)}_{L^2}$ goes to 0 as $t$ goes to 0. For $t > 0$ we write 
  \begin{equation} \label{decomp-nr-G}
    \nr{\G(t,\cdot,0)}_{L^2}^2 = \int_{\abs x \leq \sqrt t} \abs{\G(t,x,0)}^2 \, dx +  \int_{\abs x \geq \sqrt t} \abs{\G(t,x,0)}^2 \, dx.
  \end{equation}
  In these two integrals we apply \eqref{decomp-G-tx0} with $\b = \sqrt t$ and $\b = 0$, respectively. This gives
  \[
  \nr{\G(t,\cdot,0)}_{L^2}^2 \lesssim  \sqrt t  + t \int_{\abs x \geq \sqrt t} \frac 1 {\abs x^2} \, dx \lesssim \sqrt t \limt t 0 0.
  \]
  This proves the result for $\g > 0$.

  Now assume that $\g < 0$. We use the decomposition of Lemma \ref{lem-decomposition-G}. We can check that $\G_1(t,x,0)$ satisfies \eqref{decomp-G-tx0} (the additional boundary term for the integration by parts in harmless), so we conclude as above for the contribution of $\G_1$. For $\G_2$ we use the exponential decay to see that $t \mapsto \G_2(t,\cdot,0)$ is continuous from $(0,+\infty)$ to $L^2(\R)$. For the continuity at $t = 0$ we write 
  \[
  \nr{\G_2(t,\cdot,0)}_{L^2}^2 = \int_{\abs x \leq t^{1/4}} \abs{\G_2(t,x,0)}^2 \, dx +  \int_{\abs x \geq t^{1/4}} \abs{\G_2(t,x,0)}^2 \, dx
  \]
  The first term goes to 0 since $\G_2(t,x,y)$ is uniformly bounded and for the second we use the fact that for $\abs x \geq t^{1/4}$ we have $g \big(t,\abs x/(4\sqrt t) \big) \to 0$ (see Lemma \ref{lem-decomposition-G}). This concludes the proof.
\end{proof}

\subsection{\texorpdfstring{The Linear Evolution in $\Ec$}{The Linear
    Evolution in E}}

In this paragraph we extend $\G(t)$ and hence $e^{-it\Hg}$ to maps on $\Ec$.

We first recall that for $u \in \dot H^1(\R)$ (and in particular for $u \in \Ec$) there exists $C \geq 0$ such that for almost
all $x \in \R$ we have
\[
\abs {u(x)} \leq C \pppg x^{\frac 12}.
\]
Let $t \in \R$. Thanks to Lemma
\ref{lem:decay-estimate} we can define a temperate distribution $\G(t) u$ by
\begin{equation}\label{eq:distributional-sense}
  \forall \vf \in \Sc(\R) ,\quad \innp{\G(t)u}{\vf}_{\Sc',\Sc} = \innp{u}{\G(-t)\vf}.
\end{equation}
Then we can similarly extend $e^{-it\Hg}$. For $u \in \dot H^1(\R)$, the distribution $\tSgHg t u$ is defined by
\[
\forall \vf \in \Sc(\R) , \quad \innp{\tSgHg t u}{\vf} = \innp{u} {e^{it\Hg} \vf}.
\]
Of course, if $u \in H^1(\R)$ we have $\tSgHg t u = e^{-it\Hg} u \in
L^2(\R)$. As for $\Hg$, we choose a different notation to emphasize the difference between the propagator $e^{-itH_\gamma}$ defined on $L^2(\R)$ by the usual theory of selfadjoint operators and the distribution $\tSgHg{t}$ defined by duality. It will appear in the
sequel that $\tSgHg{t}$ enjoys in fact most of the properties of $e^{-itH_\gamma}$.

The following result describes the action of $\Gamma(t)$ on functions
of $\dot H^1(\R)$.

\begin{lemma}[Action of $\Gamma$ on $\dot H^1(\R)$] \label{lem-Gamma-dotH1}
  Let $T > 0$. There exists $C > 0$ such that for $t \in [-T,T]$ and $u \in \dot H^1(\R)$ we have $\G(t) u \in H^1(\R)$ and 
  \[
  \nr{\G(t) u}_{H^1} \leq C \left( \nr {u'}_{L^2} + \abs {u(0)}^2 \right).
  \]
  Moreover the map $t \mapsto \G(t) u$ is continuous from $\R$ to $H^1(\R)$.
\end{lemma}

\begin{proof}
  We first assume that $u$ vanishes on [-1,1]. With similar
  calculations as in the proof of Lemma \ref{lem-Gamma-H1} we see that for $t > 0$ and $\vf \in C_0^{\infty}(\R\setminus\{0\})$ we have
  \begin{equation} \label{eq-Gammau-phi'}
    \begin{aligned}
      \innp{\G(t) u}{\vf'}
      & = - \Re\int_\R \int_\R u(x) \partial_y \G(t,x,y)  \overline {\vf(y)} \, dy \, dx \\
      & = -\Re \int_\R \int_\R u(x) \sign(x) \partial_x \G(t,x,y) \sign(y) \overline {\vf(y)} \, dy \, dx \\
      & = \innp{ \G(t) (\sign(x) u')} {\sign(y) \vf}.
    \end{aligned}
  \end{equation}
  Since $\G(t)$ is continuous on $L^2(\R)$, this proves that $t \mapsto (\G(t)u)'$ defines a continuous map from $\R$ to $L^2(\R)$ and 
  \[
  \nr{(\G(t)u)'} \leq \nr{u'}_{L^2}.
  \]

  Now assume that $\g > 0$. 
  After an integration by parts with the operator $L_x$ defined in \eqref{def-op-L} we see that $\innp{u}{\G(-t) \vf} =\Re( A_{1}(t) + A_2(t))$ where
  \[
  A_1(t) = -i\g t \int_\R \int_\R \int_0^{+\infty}  \frac {  u'(x) \sign(x) }{s + \abs x + \abs y} e^{-\frac {\g s} 2} K_0(t, s+ \abs x + \abs y) \overline {\vf}(y) \, ds \, dy \,dx
  \]
  and
  \[
  A_2(t) =  i\g t \int_\R \int_\R \int_0^{+\infty}  \frac {u(x)}{(s + \abs x + \abs y)^2} e^{-\frac {\g s} 2} K_0(t, s+ \abs x + \abs y) \overline {\vf}(y) \, ds \, dy \,dx.
  \]
  With another integration by parts with $L_s$ we obtain 
  \begin{align} \label{estim-A1}
    \abs{A_1(t)} \lesssim t^{3/2} \int_{\abs x \geq 1} \int_{y \in \R} \frac {\abs{u'(x)} \abs{\vf(y)}}{(\abs x + \abs y)^2} \, dy \, dx \lesssim t^{3/2} \nr{u'}_{L^2} \nr \vf _{L^2}.
  \end{align}
  The term $A_2(t)$ is estimated similarly using the Hardy inequality:
  \[
  \int_\R \frac {\abs{u(x)}}{\abs x} \, dx\lesssim \nr{u' }_{L^2}.
  \]
  In all the integrals given by these two integrations by parts we can
  apply the continuity theorem under the integral sign to see that $t
  \mapsto \G(t) u$ is continuous on $(0,+\infty)$. We also see in
  \eqref{estim-A1} and the analogous estimate for $A_2$ that $\nr{\G(t)
    u}_{L^2}$ goes to 0 when $t$ goes to 0. Thus the result is proved
  for $\g > 0$ and $u$ vanishing in $[-1,1]$. 

  For the case $\g < 0$ we use the decomposition of Lemma
  \ref{lem-decomposition-G}. For $\G_1$ we proceed as in the case $\gamma>0$, and for $\G_2$ we use the exponential decay given by Lemma \ref{lem-decomposition-G} and the Hardy inequality.
  Thus we have proved the proposition if $u$ vanishes on $[-1,1]$. 

  Finally we consider the case of a generic $u$. Let $\h \in C_0^{\infty}(\R,[0,1])$ be supported in $(-2,2)$ and equal to 1 on $[-1,1]$. For $u \in \dot H^1(\R)$ we have $\h u \in H^1(\R)$ and $(1-\h) u$ vanishes on [-1,1], so with Lemma \ref{lem-Gamma-H1} we obtain that $t \mapsto \G(t) \h u + \G(t) (1-\h) u$ is continuous from $\R$ to $H^1(\R)$. Moreover for $T > 0$ fixed and $t \in [-T,T]$ we have 
  \[
  \nr{\G(t) u}_{H^1} \lesssim \nr{\h u}_{H^1} + \nr{((1-\h)u)'}_{L^2} \lesssim \nr{u'}_{L^2} + \nr{u}_{L^{\infty}([-2,2])} \lesssim \nr{u'}_{L^2} + \abs{u(0)}.
  \]
  This concludes the proof.
\end{proof}

Now we deduce from Lemma \ref{lem-Gamma-dotH1} the properties of the map $t \mapsto \tSgHg t$:

\begin{proposition}[Properties of the Propagator $\tSgHg{t}$] \label{prop-lin-evol}
  Let $u_0 \in \Ec$. The following assertions hold.
  \begin{enumerate} [(i)]
  \item For all $t \in \R$ the distribution $\tSgHg t u_0$ belongs to $\Ec$. 
  \item For $s,t \in \R$ we have $\tSgHg{s} \circ \tSgHg{t} = \tSgHg{s+t}$ on $\Ec$.
  \item The map $t \mapsto \tSgHg t u_0 - u_0$ is continuous from $\R$
    to $H^1(\R)$. 
  \item The map $t \mapsto \tSgHg t u_0$ is continuous from $\R$ to $\Ec$.
  \item
    Let $R > 0$ and $T > 0$. Then there exists $C_R > 0$ such that for $u_0 \in \Ec$ with $\abs {u_0}_{\Ec}\leq R$ and $t \in [-T,T]$ we have 
    \[
    \nr{\tSgHg t u_0 - u_0}_{H^1} \leq C_R.
    \]
  \item 
    Let $R \geq  0$ and $T \geq 0$. Then there exists $C \geq 0$ such that for $u_0,\tilde u_0 \in \Ec$ with $E(u_0) \leq R$ and $E(\tilde u_0) \leq R$ we have 
    \[
    \sup_{t \in [-T,T]}  d_\infty  \left(\tSgHg t u_0, \tSgHg t \tilde u_0\right) \leq C  d_\infty  (u_0,\tilde u_0).
    \]
  \end{enumerate}
\end{proposition}

\begin{proof}
  We first deal with the unperturbed part of the evolution. The map
  \[
  t \mapsto \tSgHo t u_0 - u_0
  \]
  is continuous from $\R$ to $H^1(\R)$. Indeed, as it was proved in
  \cite{Ge08},  it is a consequence of
  the formulation in Fourier variables:
  \begin{multline*}
    \tSgHo t u_0 -u_0=\mathcal
    F^{-1}\left(e^{it|\xi|^2}\hat u_0-\hat u_0\right)\\=\mathcal
    F^{-1}\left((-i\xi) \frac{e^{it|\xi|^2}-1}{|\xi^2|}(-i\xi)\hat u_0\right)=\mathcal
    F^{-1}\left((-i\xi) \frac{e^{it|\xi|^2}-1}{|\xi^2|}\widehat {u_0'}\right).
  \end{multline*}
  Then, thanks to Proposition \ref{prop-Gamma} and Lemma \ref{lem-Gamma-dotH1} the same holds for  
  \[
  t \mapsto \tSgHg t u_0  - u_0 = e^{-itH_0} u_0 - u_0 + \G(t)u_0.
  \]
  With Lemma \ref{lem-Ec-H1}, this proves (i), (iii) and (iv). Statement (ii) is then clear by duality. For the
  last two statements (v) and (vi), we use again the fact that they hold if $\g = 0$. The contribution of $\G(t)$ is controlled by Lemma \ref{lem-Gamma-dotH1} and Lemma \ref{lem-Ec-H1}.
\end{proof}

\subsection{\texorpdfstring{The Linear Evolution in $X_\g^2$}{The Linear Evolution in X2}}

The map $t \mapsto e^{-it\Hg} u$ is continuous for any $u \in L^2(\R)$ and is differentiable for $u \in D(\Hg)$. We expect that the map $t \mapsto \tSgHg t u_0$, continuous when $u_0 \in \Ec$, similarly enjoys better properties when $u_0 \in X^2_\g$.

\begin{proposition} [Linear Evolution in $X_\gamma^2$]\label{prop-X2} Let $t \in \R$ and $u \in X^2_\g$. Then the following properties hold.
  \begin{enumerate}[(i)]
  \item $\tSgHg t u \in X^2_\g$.
  \item $\tHg \tSgHg{t} u = e^{-it\Hg} \tHg u$.
  \item We have
    \[
    \tSgHg{t} u = u -  i \int_0^t e^{-is\Hg}  \tHg u \, ds.
    \]
    In particular, the map $t \mapsto \tSgHg{t} u$ is differentiable on $\R$ and for all $t \in \R$ we have
    \[
    \frac d {dt} \tSgHg{t} u = -i e^{-it\Hg}  \tHg u \in L^2(\R).
    \]
  \end{enumerate}
\end{proposition}

\begin{proof}
  For $t \in \R$ we set $v(t) = \tSgHg t  u$. By Proposition \ref{prop-lin-evol} we have $v(t) - u \in H^1(\R)$, so $v(t) \in L^{\infty}(\R) \cap \dot H^1(\R)$. We can write $v(t) = v_0(t) + v_\G(t)$ with $v_0(t) = \tSgHo t  u$ and $v_\G(t) = \G(t) u$.  Let $\vf \in C_0^{\infty}(\R\setminus\{0\})$. We have on the one hand
  \begin{align*}
    \innp{v_0}{\vf''}
    & = \innp{u }{ e^{itH_0} \vf''} = \innp{u }{ (e^{itH_0} \vf)''} = - \innp{u' }{ (e^{itH_0} \vf)'} \\
    & = \innp{u'' }{e^{itH_0} \vf} + \Re\left(\big(u'(0^+) - u'(0^-) \big) \big( e^{-itH_0} \overline \vf\big)(0)\right)\\
    & = \innp{u'' }{e^{itH_0} \vf} + \g \Re\left( u(0) \big( e^{-itH_0} \overline \vf\big)(0)\right).
  \end{align*}
  With the same kind of computation as in \eqref{eq-Gammau-phi'} (except that $u$ no longer vanishes on a neighborhood of 0, see also the proof of Lemma \ref{lem-Gamma-H1} in this case), we have on the other hand
  \begin{align*}
    \innp{v_\G}{\vf''}
    & = \innp{\G(t) (\sign(x) u')}{\sign(y) \vf'} + 2 \Re\left( u(0) \big( \G(t) (\sign(y) \overline \vf') \big) (0)\right)\\
    & = \innp{\G(t) u''}{\vf} + \Re\left(\g u(0) \big( \G(t) \overline \vf\big) (0) +  2 u(0) \big( \G(t) (\sign(y) \overline \vf') \big) (0)\right).
  \end{align*}
  If $\g > 0$ we have
  \begin{multline*}
    \big( \G(t) (\sign (y) \overline \vf')\big) (0) = - \frac \g 2 \int_\R \int_0^{+\infty} e^{-\frac {\g s }2} K_0(t,s+\abs y) \sign(y) \overline \vf'(y) \, ds \, dy \\
    \begin{aligned}
      & =   \frac \g 2 \int_\R \int_0^{+\infty} e^{-\frac {\g s }2} \partial _s K_0(t,s+\abs y)  \overline \vf(y) \, ds \, dy \\
      & = - \frac \g 2 \int_\R K_0(y) \overline \vf(y) \, dy +  \left( \frac \g 2 \right)^2 \int_\R \int_0^{+\infty} e^{-\frac {\g s }2}  K_0(t,s+\abs y)  \overline \vf(y) \, ds \, dy \\
      & = - \frac \g 2 \int_\R K_0(y) \overline \vf(y) \, dy - \frac \g 2 \big(\G(t) \overline \vf)(0),
    \end{aligned}
  \end{multline*}
  so finally 
  \[
  \innp{v}{\vf''} = \innp{ u''}{e^{it\Hg} \vf}.
  \]
  We obtain the same result if $\g < 0$ and, finally, we have $v'' \in L^2(\R)$ in both cases. 
  Now let $\vf_\pm \in \Sc$ be supported in $\R_\pm^*$. We have similarly
  \begin{multline*}
    - \innp{v_\G}{\vf_\pm'}
    = \pm\Re\bigg( \frac \g 2\lim_{R \to +\infty}\int_{-R}^{R} \int_\R u(x)
    K_0(t,\abs x + \abs y) \overline{\vf_\pm(y)} \, dx \, dy\\
    \quad \pm \frac \g 2 \lim_{R \to +\infty}\int_{-R}^{R} \int_\R  u(x)  \G(t,x,y) \overline {\vf_\pm (y)}  \, dx \, dy\bigg).
  \end{multline*}
  Now assume that the sequence $(\vf_n^\pm)$ of Schwartz functions
  supported in $\R_\pm^*$ is an approximation of the Dirac
  distribution. Then at the limit when $n$ goes to infinity in this equality we get 
  \[
  v'_\G(0^\pm) = \pm \frac \g 2 v(0).
  \]
  Since $v_0'(0^+)- v_0'(0^-) = 0$ this finally proves that 
  \[
  v'(0^+) - v'(0^-) = \g v(0),
  \]
  which concludes the proof of the first statement. Then $\tHg \tSgHg{t} u$ is well defined and the second statement follows by duality (against functions in $C_0^\infty(\R \setminus \{0\}$) and the fact that $e^{-it\Hg}$ and $\Hg$ commute.

  It remains to prove the last claim. For $t \in \R$ we set 
  \[
  v(t) = \tSgHg{t} u - u +i \int_0^t e^{-is\Hg} \tHg u \, ds.
  \]
  This defines a continuous function from $\R$ to $L^2(\R)$. Let $t \in \R$ and $\p_0 \in  D_0(\Hg)$. We have 
  \begin{align*}
    \innp {v(t)}{\p_0}
    = \innp{ u}{e^{it\Hg}  \p_0 -  \p_0 -i  \int_0^t e^{is\Hg}  \Hg \p_0 \, ds  }
    = 0.
  \end{align*}
  By density of $ D_0(\Hg)$ in $L^2(\R)$ we obtain that $v(t)=0$ on $\R$. Then, since the
  map $s \mapsto e^{-is\Hg}  \tHg u$ belongs to $C^0(\R,L^2(\R))$,
  the last property is proved.
\end{proof}

\section{The Cauchy Problem}
\label{sec:cauchy}

This section is devoted to the proof of Theorem
\ref{th-pb-cauchy}. We first prove that for any $u_0 \in \Ec$ the equation
\eqref{eq:gp} has a unique solution with $u(0) = u_0$. Then we study
\eqref{eq:gp} and the conservation of energy in $X^2_\gamma$. By density we obtain the conservation of energy and then
the global existence.

We first recall explicitly what is called a solution of \eqref{eq:gp}:

\begin{definition}[Solution of \eqref{eq:gp}] \label{def-sol}
  Let $u_0 \in \Ec$ and $T \in (0,+\infty]$. We say that $u : (-T,T) \to \Ec$ is a solution of \eqref{eq:gp} with $u(0) = u_0$ if the following properties are satisfied.
  \begin{enumerate} [(i)]
  \item The function $u$ is continuous from $(-T,T)$ to $(\Ec, d_\infty)$ (and hence to $(\Ec,d_0)$).
  \item We have $u(0) = u_0$.
  \item For $v \in \Sc(\R)$ we have in the sense of distributions in $(-T,T)$
    \[
    i \frac {d}{dt} \innp{u(t)}{v} - \innp{\partial_x u(t)}{\partial_x v} - \g u(t,0) \overline{v(0)} + \innp{F(u(t))}{v} = 0.
    \]
  \end{enumerate}
\end{definition}

\subsection{Local Well-Posedness in the Energy Space}

In this paragraph we prove the  local well-posedness result of \eqref{eq:gp} with initial condition in $\Ec$.
As usual for non-linear problems, it is convenient to write it in
Duhamel form.

\begin{proposition}[Duhamel Formula] \label{prop-duhamel}
  Let $u_0 \in \Ec$ and $u  \in C^0((-T,T),\Ec)$ for some $T \in (0,+\infty]$. Then $u$ is a solution of \eqref{eq:gp} with $u(0)=u_0$ if and only if 
  \begin{multline} \label{eq-duhamel}
    u(t) = \tSgHg t  u_0 + i \int_0^t e^{-i(t-s)\Hg} F(u(s))\, ds \\ = \tSgHg t  u_0  +  i \int_0^t e^{-is\Hg} F(u(t-s))\, ds.
  \end{multline}
\end{proposition}

\begin{proof}
  Since we are dealing
  with functions in $\Ec$, which is not a vector space, we have to be
  careful and check that the ideas of the standard proof indeed
  transfer to our current setting.

  We first assume that $u$ is a solution of \eqref{eq:gp}. For $t \in (-T,T)$ we set 
  \[
  \tilde u (t) = \tSgHg{-t} u(t)   - i \int_0^t e^{is\Hg} F(u(s)) \, ds .
  \]
  By Proposition \ref{prop-lin-evol}, the first term of the right-hand side defines a continuous function from $(-T,T)$ to $(\Ec,d_\infty)$. By Lemma \ref{lem-continuite-F-Ec-H1} and Proposition \ref{prop-lin-evol-H1}, the second term is of class $C^1$ with values in $H^1(\R)$, so by Lemma \ref{lem-Ec-H1} the function $\tilde u$ belongs to $C^0((-T,T),\Ec)$. Then for $v \in C_0^\infty(\R \setminus \{0\})$ we have in the sense of distributions
  \[
\frac d {dt} \innp{\tilde u (t)}{v} \, dt = 0.
  \]
  We deduce that $\tilde u$ is constant (with respect to $t$), and hence $u$ is indeed as given by \eqref{eq-duhamel}.

  Conversely, we have to check that a continuous solution of \eqref{eq-duhamel} is a solution of \eqref{eq:gp} in the sense of Definition \ref{def-sol}. The first property holds by assumption and the second is clear. 
  By Lemma \ref{lem-densite-X2}, we can find a sequence $(u_{0,n})_{n\in\N}$ of functions in $\Ec \cap X^2_\gamma$ such that $d_\infty(u_{0,n},u_0)$ goes to 0. We can also find a sequence of continuous functions $(F_n)$ from $(-T,T)$ to $D_0(\Hg)$ such that $F_n$ tends to $(F \circ u)$ in $L^\infty_{\mathrm{loc}}((-T,T),H^1(\R))$. Then for $n \in \N$ and $t \in (-T,T)$ we set 
  \[
  u_n(t) = \tSgHg t  u_{0,n} + i \int_0^t e^{-i(t-s)\Hg} F_n(s) \, ds.
  \]
  Then, by Proposition \ref{prop-X2}, the function $u_n$ belongs to $C^0 \big( (T,T),\Ec \big)$, is differentiable with $\partial_t u_n \in C^0 \big( (T,T),L^2(\R) \big)$ and for $v \in \Sc(\R)$
  \[
  i \frac {d}{dt} \innp{u_n}{v} = \innp{\tHg u_n}{v} - \innp{F_n(t)}{v} = \innp{\partial_x u_n}{\partial_x v} + \g u_n(t,0) \overline{v(0)} - \innp{F_n(t)}{v}.
  \]
  Now for $\vf \in C_0^\infty(-T,T)$ we multiply this equality by $\vf(t)$, take the integral over $t \in (-T,T)$, perform a partial integration on the left-hand side and take the limit $n \to \infty$ to conclude.
\end{proof}

Now we can prove the local well-posedness of \eqref{eq:gp} and the continuity with respect to the initial condition:

\begin{proposition}[Local Well-Posedness] \label{prop-loc-cauchy}
  Let $R > 0$. Then there exists $T > 0$ such that for all $u_0 \in \Ec$ with $\abs{u_0}_{\Ec} \leq R$ the problem \eqref{eq:gp} has a unique solution $u : (-T,T) \to \Ec$ with $u(0) = u_0$. Moreover there exists $C_R > 0$ such that for $u_0,\tilde u_0 \in \Ec$ with $\abs{u_0}_\Ec \leq R$ and $\abs{\tilde u_0}_\Ec \leq R$ then the corresponding solutions $u$ and $\tilde u$ satisfy
  \[
  \forall t \in (-T,T), \quad d_\infty \big( u(t), \tilde u (t) \big) \leq C_R d_\infty (u_0,\tilde u_0).
  \]
\end{proposition}

\begin{proof}
  Let $u_0 \in \Ec$,  $T>0$, and $w \in C^0((-T,T),H^1(\R))$. By Proposition \ref{prop-lin-evol}, Lemma \ref{lem-Ec-H1}, Lemma \ref{lem-continuite-F-Ec-H1} and Proposition \ref{prop-lin-evol-H1} the function 
  \[
  s \mapsto e^{-i(t-s)\Hg} F \big( w(s) + \tSgHg{s} u_0 \big) 
  \]
  belongs to $C^0((-T,T),H^1(\R))$ for all $t \in (-T,T)$. Thus we can set
  \begin{equation} \label{def-Phi}
    \Phi_{T,u_0}(w) : t \mapsto  i \int_0^t e^{-i(t-s)\Hg} F \big( w(s) +  \tSgHg{s} u_0\big) \, ds.
  \end{equation}
  This also defines a function in $C^0((-T,T),H^1(\R))$.

  Given $u \in C^0((-T,T),\Ec)$, the equality \eqref{eq-duhamel} is then equivalent to 
  \begin{equation} \label{eq-w-fixed-point}
    w = \Phi_{T,u_0}(w)
  \end{equation}
  where we have set 
  \begin{equation} \label{def-w}
    w : t \mapsto   u(t) - \tSgHg{t} u_0.
  \end{equation}
  Our purpose is to use the fixed point Theorem to prove that \eqref{eq-w-fixed-point} has a unique solution in a suitable space.

  Let $R > 0$ be greater that $\abs{u_0}_\Ec$. For $T > 0$ we set 
  \[
  \WW = \left \{ w \in C^0((-T,T),H^1(\R)) \st \nr{w(t)}_{H^1} \leq R \text{ for all } t \in (-T,T) \right \}.
  \]

  By Lemma \ref{lem-Ec-H1} and Proposition \ref{prop-lin-evol} there exists $\tilde R$ which only depends on $R$ such that for $w \in \WW$ and $s \in (-T,T)$ we have
  \[
  \abs{ w(s) + \tSgHg s u_0}_\Ec \leq \tilde R.
  \]
  Then, by Proposition \ref{prop-lin-evol-H1}, Lemma \ref{lem-continuite-F-Ec-H1} and Proposition \ref{prop-lin-evol} we have for all $T > 0$ and $t \in (-T,T)$
  \begin{align*}
    \nr{\Phi_{T,u_0}(w)(t)}_{H^1} \lesssim T \sup_{s \in (-T,T)} \nr{F\big( w(s) + \tSgHg{s}u_0 \big)}_{H^1} \lesssim T.
  \end{align*}
  This proves that if $T > 0$ is small enough then we have 
  \[
  \nr{\Phi_{T,u_0}(w)}_{L^{\infty}((-T,T),H^1)} \leq R.
  \]
  We similarly prove that for $T > 0$ small enough we have
  \begin{equation} \label{estim-PhiT}
    \nr{\Phi_{T,u_0}(w) - \Phi_{T,u_0}(\tilde w)}_{L^{\infty}((-T,T),H^1)} \leq \frac 12  \nr{w - \tilde w}_{L^{\infty}((-T,T),H^1)}.
  \end{equation}
  In particular, for $T > 0$ small enough, $\Phi_{T,u_0}$ is a contraction of $\WW$. Now let such a $T$ be fixed. By the fixed point theorem there exists a solution $w \in \WW$ of \eqref{eq-w-fixed-point}, which gives a solution $u$ of \eqref{eq:gp} with $u(0) = u_0$. Conversely, if $u$ is such a solution on $(-T,T)$ for some $T > 0$, then  $w$ given by \eqref{def-w} belongs to $\WW$ for $R$ large enough. We deduce uniqueness.

  Finally, we prove the continuity of $u(t)$ with respect to $u_0$. Let $u_0,\tilde u_0 \in \Ec$ and $R > 0$ be such that $\abs{u_0}_\Ec \leq R$ and $\abs {\tilde u_0}_\Ec \leq R$. Let $w,\tilde w \in \WW$ be the fixed points for $\Phi_{T,u_0}$ and $\Phi_{T,\tilde u_0}$ respectively, $T > 0$ being chosen small enough. As for \eqref{estim-PhiT} we see that for $T > 0$ smaller if necessary we have
  \begin{align*}
    \nr{w- \tilde w }_{L^{\infty}((-T,T),H^1)}
    & = \nr{\Phi_{T,u_0}(w)- \Phi_{T,\tilde u_0}(\tilde w)}_{L^{\infty}((-T,T),H^1)}\\
    & \leq \frac 12 \left( d_\infty (u_0,\tilde u_0) + \nr{w- \tilde w }_{L^{\infty}((-T,T),H^1)} \right),
  \end{align*}
  and hence 
  \[
  \nr{w- \tilde w }_{L^{\infty}((-T,T),H^1)} \leq  d_\infty (u_0,\tilde u_0).
  \]
  With \eqref{def-w} and Proposition \ref{prop-lin-evol}, we obtain that for all $t \in (-T,T)$ we have 
  \[
  d_\infty  (u(t),\tilde u(t)) \lesssim  d_\infty  (u_0,\tilde u_0),
  \]
  and the proposition is proved.
\end{proof}

From Proposition \ref{prop-loc-cauchy} we deduce the following result.
\begin{corollary} \label{cor-loc-cauchy}
Let $u_0 \in \Ec$. Then the problem \eqref{eq:gp} has a unique maximal solution $u$ with $u(0) = u_0$, defined on $(-T_-,T_+)$ for some $T_\pm \in (0,+\infty]$. Moreover if $T_\pm < +\infty$ then 
\[
\abs{u(t)}_\Ec \limt {t} {\pm T_\pm} +\infty.
\]
\end{corollary}

\begin{proof}
  Let $u_0\in \Ec$ and let $u$ be the unique maximal solution of
  \eqref{eq:gp} with $u(0) = u_0$, defined on $(-T_-,T_+)$. Assume by
  contradiction that $T_+<+\infty$ and there exists $R>0$ such that
  for every $n\in\mathbb N$ there exists $t_n\in (T_+-1/n,T_+)$ with 
  $\abs{u(t_n)}_\Ec<R$. Let $T$ be the time of existence given by
  Proposition \ref{prop-loc-cauchy} and let $n$ be such that
  $t_n+T>T_+$. By Proposition
  \ref{prop-loc-cauchy}, $u$ exists on $(t_n-T,t_n+T)$. However, since
  $t_n+T>T_+$, we have a contradiction with the maximality of $T_+$. The same
  reasoning works with $T_-$. 
  Therefore  if $T_\pm < +\infty$ then 
\[
\abs{u(t)}_\Ec \limt {t} {\pm T_\pm} +\infty,
\]
which is the desired result. 

\end{proof}

\subsection{Conservation of Energy and Global Existence}

In order to prove the conservation of the energy, we need a solution
of \eqref{eq:gp} in a strong sense. This is the case when the  initial
condition is in $X^2_\g$:

\begin{proposition}[Local Well-Posedness at High Regularity] \label{prop-cauchy-X2}
Let $u_0 \in \Ec \cap X^2_\g$ and let $u$ be the maximal solution of
\eqref{eq:gp} with $u(0) = u_0$, as given by Corollary
\ref{cor-loc-cauchy}. Let $(-T_-,T_+)$ be the interval of definition
of $u$, with $T_\pm \in (0,+\infty]$. Then for all $t\in (-T_-,T_+)$,
$u(t)\in X^2_\g$, $u$ is differentiable with $\partial_t u \in C^0((T_-,T_+),L^2(\R))$ and for all $t \in (-T_-,T_+)$ we have
\begin{equation} \label{eq-u-reg}
\partial_t u(t) = -i \tHg u(t) + i F(u(t)).
\end{equation}
\end{proposition}

\begin{proof} 
Let $R > 0$. We prove that there exists $\tilde T > 0$ such that for
all $u_0 \in \Ec \cap X^2_\g$ with $\abs {u_0}_\Ec \leq R$ and $\| \tHg u_0
\|_{L^2(\R)} \leq R$ the maximal solution $u$ of \eqref{eq:gp} is at
least defined on $(-\tilde T,\tilde T)$, belongs $X^2_\g$ for all $t\in (-T_-,T_+)$,
and is differentiable with $u_t\in C^0((-\tilde T,\tilde T),L^2(\R))$.

Let $\tilde T > 0$ and $\tilde R > 0$. We denote by $\WWWbis$ the set of functions $\tilde w \in C^0((-\tilde T, \tilde T),H^1(\R)) \cap C^1((-\tilde T, \tilde T),L^2(\R))$ such that $\tilde w(0) = 0$ and, for all $t \in (-\tilde T, \tilde T)$,
\[
\nr{\tilde w(t)}_{H^1} \leq \tilde R  \quad \text{and} \quad  \nr{\partial_t \tilde w(t)}_{L^2} \leq \tilde R.
\]
Let $\tilde w \in \WWWbis$. For $t \in (-\tilde T,\tilde T)$ we set
$v(t) = \tilde w(t) + \tSgHg{t} u_0$. By Proposition \ref{prop-X2},
$v$ is differentiable with $v_t\in C^0((-\tilde T,\tilde T),L^2(\R))$.
Then $(F \circ v)$ belongs to $C^1((-\tilde T,\tilde T),L^2(\R))$ with 
\begin{equation} \label{der-F-circ-v}
\partial_t (F \circ v)(t) = v_t-2|v|^2v_t-v^2\bar v_t. 
\end{equation}
For $t \in (-\tilde T,\tilde T)$ we have (recall that $\Phi$ was defined in \eqref{def-Phi}) 
\begin{multline*}
\frac {(\Phi_{\tilde T,u_0}(\tilde w))(t+\t) - (\Phi_{\tilde T,u_0}(\tilde w))(t)}\t \\
= i \int_0^t e^{-is\Hg} \frac {F(v(t+ \t -s)) - F(v(t-s))} \t \, ds + \frac {i } \t \int_t^{t+\t} e^{-is\Hg} F(v(t+\t -s)) \, ds.
\end{multline*}
Taking the limit $\t \to 0$ we obtain that $\Phi_{\tilde
    T,u_0}(\tilde w)$ is continuously differentiable with $\partial_t
  (\Phi_{\tilde T,u_0}(\tilde w))\in C^0((-\tilde T,\tilde T),L^2(\R))$ and for $t \in (-\tilde T,\tilde T)$ we have
\[
\partial_t (\Phi_{\tilde T,u_0}(\tilde w))(t) = i \int_0^t
e^{-is\Hg} \partial_t(F \circ v)(t-s) \, ds + i e^{-it\Hg}
F(u_0).
\]
In particular,
\[
\nr{\partial_t(\Phi_{\tilde T,u_0}(\tilde w))(t)}_{L^2}
\lesssim C_{R} (1+ \tilde T C_{\tilde R}),
\]
where $C_{R}$ only depends on $R$ and $C_{\tilde R}$ only depends on $\tilde R$. Moreover for $\tilde w_1,\tilde w_2 \in \WWWbis$ we have 
\begin{multline*}
\nr{\partial_t(\Phi_{\tilde T,u_0}(\tilde w_1))(t)
  - \partial_t(\Phi_{\tilde T,u_0}(\tilde w_2))(t)}_{L^2} \\\leq \tilde
T C_{R} C_{\tilde R}\left(\norm{\tilde w_1-\tilde
    w_2}_{L^\infty((-\tilde T,\tilde T),H^1)}+\norm{\partial_t\tilde w_1-\partial_t\tilde
    w_2}_{L^\infty((-\tilde T,\tilde T),L^2)}\right).
\end{multline*}
Finally $(\Phi_{\tilde T,u_0}(\tilde w))(0) = 0$, so for $\tilde R$ large enough and $\tilde T$ small enough the map $\Phi_{\tilde T,u_0}$ defines a contraction of $\WWWbis$. Thus the equation $\Phi_{\tilde T,u_0} \tilde w =\tilde w$ has a unique solution in $\WWWbis$. By uniqueness, this proves that the fixed point $w$ of $\Phi_{T,u_0}$ obtained in the proof of Theorem \ref{th-pb-cauchy} is in $\WWWbis$. 

Let $t \in (-\tilde T,\tilde T)$. We have
\begin{eqnarray*}
  \lefteqn{\frac {e^{-i\t \Hg} - 1}{\t} (\Phi_{\tilde T,u_0}(w))(t) }\\
&& = \frac {(\Phi_{\tilde T,u_0}(w))(t + \t) - (\Phi_{\tilde T,u_0}(w))(t)}{\t} - \frac i \t \int_t^{t+\t} e^{-i(t+\t -s)\Hg} F(v(s))\,ds\\
&& \limt \t 0 \partial_t(\Phi_{\tilde T,u_0}(w))(t) - i F(v(t)).
\end{eqnarray*}
This proves that $w(t) = (\Phi_{\tilde T,u_0}(w))(t) \in  D(\Hg)$ with 
\begin{equation*} 
  -i \Hg (\Phi_{\tilde T,u_0}(w))(t)  =\partial_t (\Phi_{\tilde T,u_0}(w))(t) - i F(v(t)).
\end{equation*}
Therefore the solution of \eqref{eq:gp} satisfies \eqref{eq-u-reg} on $(-\tilde T,\tilde T)$.

By uniqueness of a solution and the fact that the time $\tilde T$ only
depends on $R$ above, we obtain for $u_0 \in X^2_\g$ and $T_-,T_+$
given by Corollary \ref{cor-loc-cauchy} a maximal interval $(-\tilde
T_-,\tilde T_+)$ (with $\tilde T_\pm \in (0,T_\pm]$) such that the
solution $u$ of \eqref{eq:gp} lives in $X^2_\g$, is differentiable with $u_t \in C^0((-\tilde T_-,\tilde T_+),L^2(\R))$ and satisfies \eqref{eq-u-reg} on $(-\tilde T_-,\tilde T_+)$. Moreover if $\tilde T_\pm < +\infty$ we have 
\begin{equation} \label{eq-blowup-reg}
\abs{u(t)}_\Ec + \| \tHg u(t) \|_{L^2(\R)} \limt {t} {\pm \tilde T_\pm} +\infty.
\end{equation}

Now assume that $\tilde T_+ < T_+$. Then by continuity of $u$ in $\Ec$ on $[0,T_+)$ we obtain that $\abs{u}_\Ec$ is bounded on $[0,\tilde T_+)$. By \eqref{eq-duhamel} we have for $t \in [0,\tilde T_+)$
\[
\partial_t u(t) = -i \tSgHg t \tHg u_0 + i e^{-it\Hg} F(u_0) + i \int_0^t e^{-is\Hg} \partial_t (F\circ u) (t-s) \, ds.
\]
The first two terms are bounded on $[0,\tilde T_+)$. Since $\abs{u(t)}_\Ec$ is also bounded, we obtain with \eqref{der-F-circ-v} applied to $u$ that there exists $C > 0$ such that 
\[
\nr{\partial_t u(t)}_{L^2(\R)} \leq C + C \int_0^t \nr{\partial_t u(s)}_{L^2(\R)}.
\]
By the Gr\"onwall Lemma, $\partial_t u$ is bounded in $L^2(\R)$ on the bounded interval $[0,\tilde T_+)$. By \eqref{eq-u-reg}, $\tHg u(t)$ is also bounded, which gives a contradiction with \eqref{eq-blowup-reg} and concludes the proof.
\end{proof}

We are now in position to finish the proof of Theorem
\ref{th-pb-cauchy}, \ie to prove the conservation of the energy $E_\g$ and the global existence for the solution of \eqref{eq:gp}.

\begin{proof}[Proof of Theorem \ref{th-pb-cauchy}, Global Existence, assuming Conservation of Energy]  
Take an initial data $u_0 \in \mathcal E$. Let $u$ be the maximal solution of \eqref{eq:gp} with $u(0) = u_0$. It is defined on some interval $I$ of $\R$. By conservation of energy and the energy
bound of Lemma \ref{lem-abs-Ec}, there exists $R \geq 0$ such that 
$
\abs{u(t)}_{\Ec} \leq R
$
for all $t \in I$. By Corollary \ref{cor-loc-cauchy}, this proves that $I = \R$.
\end{proof}

\begin{proof}[Proof of Theorem \ref{th-pb-cauchy}, Conservation of Energy]
Let $u_0 \in \Ec$ and let $u$ be the maximal solution of \eqref{eq:gp} with $u(0) = u_0$. It is defined on some interval $I$ of $\R$. 

If $u_0 \in X^2_\g$ then by Proposition \ref{prop-cauchy-X2} the map $t \mapsto E_\gamma(u(t))$ is differentiable on $I$ with derivative 0, so $E_\g(u(t))$ is constant on $I$ (and hence $I = \R$). The theorem is proved in this case.

Even if $u_0$ is not in $X^2_\g$, there exists a sequence $(u_{0,n})_{n \in \N}$ of functions in $\Ec \cap X^2_\g$ which converges to $u_0$ in $\Ec$. For all $n \in \N$ we denote by $u_n$
  the maximal solution of \eqref{eq:gp} with initial condition $u_{0,n}$. By the global existence result for $u_{0,n}$, $u_n$ is defined on $\R$, and in particular on $I$. By continuity of the flow in $\Ec$ and the continuity of the energy (see Lemma \ref{lem-E-continue}) we have for all $t \in I$:
  \[
  E_\g (u(t)) = \lim_{n \to+\infty} E_\g(u_n(t)) = \lim_{n \to+\infty} E_\g (u_{0,n}) = E_\g(u_0).
  \]
Thus we have conservation of the energy for $u$, which is then globally defined. This concludes the proof of Theorem \ref{th-pb-cauchy}.
\end{proof}

\section{Existence and Characterizations of Black Solitons}
\label{sec:existence}

\subsection{Existence  of Black Solitons}

As announced in introduction, the finite energy stationary solutions
to \eqref{eq:gp} are given in the following result. 

\begin{proposition}[Existence of Black Solitons]\label{prop:existence}
  Let $\gamma\in\R\setminus\{0\}$. Then the set of  finite-energy solutions to \eqref{SGPdelta} is
{
  \begin{gather*}
    \left\{e^{i\theta}\kappa,e^{i\theta}b_\gamma\ :\
      \theta\in\R\right\},\quad \text{if }\gamma>0\\
    \left\{e^{i\theta}\kappa,e^{i\theta}b_\gamma,e^{i\theta}\tilde b_\gamma\ :\
      \theta\in\R\right\},\quad \text{if }\gamma<0
  \end{gather*}
  where
  \[
  \kappa(x):=\tanh \left( \frac{x}{\sqrt{2}}\right),\;
  b_{\gamma}(x):=\tanh \left(
    \frac{|x|-c_{\gamma}}{\sqrt{2}}\right),\;
 \tilde b_{\gamma}(x):=\coth \left( \frac{|x|+c_{\gamma}}{\sqrt{2}}\right),
  \]
  for $c_{\gamma}:=\frac{1}{\sqrt{2}}\sinh^{-1}\left(
    -\frac{2\sqrt{2}}{\gamma} \right)$. 
}
\end{proposition}

Some preparation is in order. We first recall that $u \in \Ec$ is said to be a solution of \eqref{SGPdelta} if for all $\vf \in C_0^\infty(\R)$ we have 
\begin{equation} \label{D'solution}
  \int_{\R}u' \overline \vf' + \g u(0) \overline \vf(0) - \int_{\R} (1-|u|^2)u \overline \vf  =0.
\end{equation}
By elliptic regularity, such solutions are in fact smooth, except at
the origin, where they satisfy the jump condition.

\begin{lemma}\label{PropositionPropertiesOfSolutions}
  Let $\gamma\in\R\setminus\{0\}$ and $u \in \Ec$ be a solution of
  \eqref{SGPdelta}. Then
  \begin{eqnarray*}
    &u\in  C^{\infty}(\R\setminus\{0\})\cap C^0(\R),\label{regularity}\\
    & u''+(1- \abs u^2)u=0 \quad \text{on } \R \setminus \{0\},\label{eqhalfline}\\
    & u'(0^+)-u'(0^-)=\gamma u(0). \label{jumpCondition}
  \end{eqnarray*}
\end{lemma}

\begin{proof} 
  The continuity of $u$ is given by Lemma \ref{lem-abs-Ec}. From \eqref{D'solution} applied with $\vf \in C_0^\infty(\R \setminus \{0\})$ we deduce that  $u\in\mathcal{E}$ is a solution of
  \begin{equation} \label{eq-hors-0}
    u''+(1-|u|^2)u=0
  \end{equation}
  in the sense of distributions on $\R \setminus \{0\}$. This implies that $u$ is in fact smooth and a classical solution of this equation on $\R\setminus\{0\}$. Finally, we consider $\vf \in C_0^\infty(\R)$ with $\vf(0) = 1$. Starting from \eqref{D'solution} and using \eqref{eq-hors-0} after an integration by parts gives the jump condition and concludes the proof of the lemma.
\end{proof}

Let us now determine what are the finite energy solutions on the half-line.

\begin{lemma}\label{PropositionExplicitSolutionGamma=0} Assume that $u \in \Ec$ is a solution to
  \begin{equation}\label{SGP}
    u''+ (1-|u|^2)u=0,\quad  \text{on}\quad(0,+\infty).
  \end{equation}
  Then there exist $\theta\in\R$ and $c\in\R$ such that either $u(x) = e^{i\th}$ for all $x\in(0,+\infty)$, or
{
  \begin{equation} \label{sol-tanh}
\begin{cases}    
&\forall x \in (0,+\infty), \quad   u(x)= e^{i\theta}
\tanh\left(\frac{x- c}{\sqrt{2}} \right),\\
c<0,&\forall x \in (0,+\infty), \quad   u(x)= e^{i\theta}
\coth\left(\frac{x- c}{\sqrt{2}} \right).\\
\end{cases}
  \end{equation}
}
  The same conclusion holds if we replace $(0,+\infty)$ by
  $(-\infty,0)$ {and $c<0$ by $c>0$}.
\end{lemma}

\begin{proof}
  Equation \eqref{SGP} may be integrated using standard arguments from ordinary differential equations, which we recall now.

  Multiplying the equation by $ u'$ and taking the real part
  we obtain
  \[
  \frac{d}{dx}\left(\frac{1}{2}|u'|^2-\frac{1}{4}\left(1-|u|^2 \right)^2 \right)=0,
  \]
  so there exists $K\in\R$ such that
  \[
  \frac{1}{2}|u'|^2-\frac{1}{4}\left(1-|u|^2 \right)^2\equiv K.
  \]
  By Lemma \ref{lem-abs-Ec}, $\frac{1}{2}|u'(x)|^2$ goes to $K$ as $x$ goes to $+\infty$. Since $u' \in L^2(\R)$, we necessarily have $K=0$, so
  \begin{equation}\label{ide}
    \frac{1}{2}|u'|^2-\frac{1}{4}\left(1-|u|^2 \right)^2\equiv 0.
  \end{equation}
  If $|u(x_0)|=1$ for some $x_0 > 0$, then $|u'(x_0)|=0$ and by uniqueness we have $u\equiv C$ where $|C|=1$. Now we assume that $|u(x)|\neq 1$ for every $x\in(0,+\infty)$.
  Since $|u(x)|$ goes to 1 as $x$ goes to $+\infty$,  there exists $A\geq 0$ such that $|u(x)|>0$ for $|x|>A$. Therefore  we may write $u(x):=e^{i\theta(x)}\rho(x)$ for $x> A$, where $\theta,\rho\in  C^2$ and $\rho>0$. Writing down the system of equations satisfied by $\theta$ and $\rho$ we see in particular that 
  \[
  \theta''\rho+2\theta'\rho'\equiv 0\quad\text{ for }x>A,
  \]
  which implies that 
  \[
  \frac{d}{dx}(\theta'\rho^2)\equiv 0\quad\text{ on } (A,+\infty).
  \]
  Therefore there exists $\widetilde K\in\R$ such that $\theta'\rho^2\equiv \widetilde K$ for $x>A$ .
  Since $\rho(x)\rightarrow 1$, as $x\rightarrow +\infty$, it follows that $\theta'(x)\rightarrow \widetilde K$, and hence
  \[
  |u'(x)|^2=(\rho')^2+(\theta')^2\rho^2\ \geq\ (\theta')^2\rho^2 \limt x {+\infty} \widetilde K^2.
  \]
  As above it follows that $\widetilde K=
  0 $. As a consequence $\theta'\equiv 0$ on $(A,+\infty)$, so there exists $\theta_0\in\R$ such that 
  \[
  \forall x > A, \quad u(x) = e^{i\theta_0}\rho(x).
  \]
  Since $|u|\neq 1$ on $(0,+\infty)$, we infer from \eqref{ide} that
  \[
  \frac{\rho'}{(1-\rho^2) }=\pm\frac{1}{\sqrt{2}} \quad \text{on }(A,+\infty).
  \] 
  By explicit integration, there exists $c \in \R$ such that for $x > A$ we have
  \[
  \text{ either }  \rho(x)=\tanh \left( \pm\frac{x-c}{\sqrt{2}}\right)
    \text{ or } \rho(x)=\coth \left( \pm\frac{x-c}{\sqrt{2}}\right),\;c<A.
  \]
  Since $\tanh$ {and $\coth$ are} odd, up to replacing $\theta_0$ by $\theta_0+\pi$ we have either
  \[
  u(x) = e^{i\theta_0} \tanh\left(\frac{x-c}{\sqrt{2}} \right) \quad \text{on }(A,+\infty),
  \]
  or
   \[
  u(x) = e^{i\theta_0} \coth\left(\frac{x-c}{\sqrt{2}} \right),\ c<A, \quad \text{on }(A,+\infty).
  \]
  By the Cauchy Lipschitz Theorem we can take $A=0$.
\end{proof}

\begin{proof}[Proof of Proposition \ref{prop:existence}]
  Let $u$ be a  finite-energy solution to \eqref{SGPdelta}. From Lemma \ref{PropositionPropertiesOfSolutions} and from the characterization given by Lemma \ref{PropositionExplicitSolutionGamma=0}, $u$ is either constant with modulus 1 or of the form \eqref{sol-tanh} on each side of the origin. Assume by contradiction that $u$ is constant on $(-\infty,0)$. By continuity, $\abs{u(0)} = 1$ and $u$ is also constant on $(0,+\infty)$. This gives a contradiction with the jump condition. Thus $u$ is of the form \eqref{sol-tanh} on $(-\infty,0)$. By continuity (or by a similar argument), $u$ is also of the form \eqref{sol-tanh} on $(0,+\infty)$.
{
More precisely, there exist $\th_-,\th_+,c_-,c_+ \in \R$ such that for $\pm x > 0$ we have either
  \begin{equation}\label{eq:case1}
  u(x) = e^{i\th_\pm} \tanh \left( \frac {x - c_\pm} {\sqrt 2}
  \right),
  \end{equation}
or
\begin{equation}\label{eq:case2}
c_+<0,\, c_->0,\,u(x)=e^{i\th_\pm} \coth \left( \frac {x - c_\pm} {\sqrt 2}  \right).
\end{equation}
}

Assume first that \eqref{eq:case1} holds.
  By continuity at the origin we have $e^{i\th_+} = e^{i\th_-}$ or $e^{i\th_+} = - e^{i\th_-}$. In the first case we necessarily have $c_+ = c_-$. And with the jump condition we see that in fact $c_+ = c_- = 0$, so
  \begin{equation}    \label{FirstSolution}
    \forall x \in \R, \quad u(x) = e^{i\th_+} \tanh\left( \frac x {\sqrt 2} \right).
  \end{equation}
  If $e^{i\th_+} = -e^{i\th_-}$ then by continuity we have $c : = c_+ = -c_-$. Thus
  \begin{equation} \label{SecondSolution}
    \forall x \in \R, \quad u(x) = e^{i\th_+} \tanh\left( \frac {\abs x - c} {\sqrt 2} \right).
  \end{equation}
  Since $u$ is even, the jump condition reads $2u'(0^+)=\gamma u(0)$. More explicitly we have
  \[
  \sqrt{2} \left(\cosh \left( \frac{-c}{\sqrt{2}}\right)\right)^{-2}=\gamma \tanh \left( \frac{-c}{\sqrt{2}}\right),
  \]
  so
  \[
  \sqrt{2}=\gamma\sinh \left( \frac{-c}{\sqrt{2}}\right)\cosh \left( \frac{-c}{\sqrt{2}}\right)=\frac{\gamma}{2}\sinh \left( -\sqrt{2}c\right)=-\frac{\gamma}{2}\sinh \left( \sqrt{2}c\right),
  \]
  and finally
  \begin{equation}
    \label{c_gamma}
    c =   c_{\gamma}:=\frac{1}{\sqrt{2}}\sinh^{-1}\left(-\frac{2\sqrt{2}}{\gamma} \right).
  \end{equation}
  Note that $c_{\gamma}>0$ if $\gamma<0$ and $c_{\gamma}<0$ if
  $\gamma>0$.

Assume now that \eqref{eq:case2} holds.
By continuity at the origin we again have $e^{i\th_+} = e^{i\th_-}$ or
$e^{i\th_+} = - e^{i\th_-}$, but this time, due to the singularity of
$\coth$ at $0$, only $e^{i\th_+} = -
e^{i\th_-}$ is possible. Arguing as previously, we find that
\eqref{eq:case2} is possible only if $\gamma<0$, and in that case
  \begin{equation} \label{ThirdSolution}
    \forall x \in \R, \quad u(x) = e^{i\th_+} \coth\left( \frac {\abs x +c_\gamma} {\sqrt 2} \right).
  \end{equation}
where $c_\gamma$ is as in \eqref{c_gamma}.

  In conclusion the functions given by \eqref{FirstSolution},
  \eqref{SecondSolution}, \eqref{c_gamma}, \eqref{ThirdSolution} are the only candidates to be
  finite-energy solutions to \eqref{SGPdelta}. Conversely we
  can verify directly that this is indeed the case, which concludes the proof. 
\end{proof}

\subsection{Variational Characterizations}

This section is devoted to the proof of Proposition
\ref{prop:minimization}. 
Let us recall that for  $\gamma\in\R\setminus\{ 0\}$ we have set
\[
m_{\gamma}:=\inf \{E_{\gamma}(v): v \in \Ec\} > - \infty,
\] 
and that we want to prove that the infimum is achieved at solutions to
\eqref{SGPdelta}. Precisely, we want to prove that
\[
\mathcal G_\gamma=\{ e^{i\theta}b_\gamma, \theta\in\R \},\text{ if }\gamma>0;\quad
\mathcal G_\gamma=\{ e^{i\theta}\tilde b_\gamma, \theta\in\R \},\text{ if }\gamma<0,
\]
where  we have defined 
\[
\mathcal G_\gamma:=\{ v\in\mathcal E,\quad E_\g(v)=m_\gamma \}.
\]
Finally, we also want to prove compactness of the minimizing
sequences, i.e. any minimizing sequence $(u_n)\subset \mathcal E$ such that $E(u_n)\to m_\gamma$ verifies, up to a subsequence,
\[
d_0(u_n,\mathcal G_\gamma)\to 0.
\]

\begin{proof}[Proof of Proposition \ref{prop:minimization}]
  We first remark that by Lemma \ref{lem-abs-Ec} the energy is bounded from below, so $m_\g$ is finite. Let $(v_n)_{n\in\mathbb N} \subset \Ec$ be a minimizing sequence, \ie
  \[
  E_{\gamma}(v_n)=\frac{1}{2}\norm{v_n'}_{L^2}^2+\frac{\gamma}{2}|v_n(0)|^2+\frac{1}{4}\norm{1-|v_n|^2}_{L^2}^2\rightarrow m_{\gamma}.
  \]
  By Lemma \ref{lem-abs-Ec} again, the sequence $(v'_n)$ is bounded in $L^2(\R)$. Since $L^2(\R)$ is a reflexive Banach space, there exists $g\in L^2(\R)$ such that, up to a subsequence, $v'_n\rightharpoonup g$ weakly in $L^2(\R)$. On the other hand, the sequence $(u_n(0))$ is also bounded, so $v_n$ is uniformly bounded in $H^1(I)$ for every bounded interval $I\subset\R$. Hence by Rellich compactness theorem there exists $f\in L^{\infty}_{\mathrm{loc}}(\R)$ such that, up to a subsequence, $v_n\rightarrow f$ in $L^{\infty}_{\mathrm{loc}}(\R)$. Since $H^1(I)$ is a reflexive Banach space,  there exists $u\in H^1_{\mathrm{loc}}(\R)$ such that up to a subsequence $v_n\rightharpoonup u$ in $H^1_{\mathrm{loc}}(\R)$.
  But then $g=u'\in L^2(\R)$, and $f=u$. Finally,
  \[
  v_n'\rightharpoonup u'\ \mbox{ in }\ L^2(\R) \quad \mbox{ and }\quad  v_n\rightarrow u \mbox{ in }L^{\infty}_{\mathrm{loc}}(\R).
  \] 
  By the weak-lower semicontinuity of the $L^2(\R)$-norm and Fatou lemma we have
  \begin{align*}
    E_{\gamma}(u)
    & =\frac{1}{2}\norm{u'}_{L^2}^2+\frac{\gamma}{2}\lim_{n\to+\infty}| v_n(0)|^2+\frac{1}{4}\int_{\R}\liminf_{n\to+\infty}\left(1-|v_n|^2\right)^2dx\\
    &\leq \liminf_{n\to+\infty}E_{\gamma}(v_n),
  \end{align*}
  so that $E_\g(u) = m_\g$. In particular $u\in\mathcal{E}$, and we easily see that $v_n\to u$ in $(\mathcal E,d_0)$.

  Now we show that this limit $u$ is a solution of \eqref{SGPdelta}. Let $\vf \in C_0^\infty(\R)$ and $t\in \R$. We have
  \begin{align*}
    0 
    & \leq \liminf_{t \to 0} \frac {E_\g(u + t \vf) - E_\g(u)} t\\ 
    & \leq \Re \left(\int_{\R} u'\bar\vf' dx + \gamma  u(0)\bar\vf(0) -\int_{\R}\left(1-|u|^2\right)u\bar\vf dx
      \right).
  \end{align*}
  Since the choice of $\vf$ is arbitrary (we can replace $\vf$ by $-\vf$ or $\pm i \vf$) we get for all $\vf \in C_0^\infty(\R)$
  \[
  \int_{\R} u'{\bar\vf}' dx+ \gamma  u(0){\bar\vf}(0) -\int_{\R}\left(1-|u|^2\right)u\bar\vf \, dx = 0.
  \]
  This is \eqref{D'solution}, which means that $u$ is a solution of \eqref{SGPdelta}.

  By Proposition \ref{prop:existence} there exists $\theta\in\R$ such that
  either $u=e^{i\theta}b_{\gamma}$, or $u=e^{i\theta}\tilde
    b_\gamma$, or $u=e^{i\theta}\kappa$.
  To conclude it is enough to show that:
  \begin{equation} \label{comparaison-energies}
    \begin{cases}
      E_{\gamma}(b_{\gamma})< E_{\gamma}(\kappa)&\text{ if } \gamma>0,\\
      E_{\gamma}(b_{\gamma})> E_{\gamma}(\kappa) >E(\tilde b_\gamma)&\text{ if } \gamma<0.
    \end{cases}
  \end{equation}
Since $\kappa$, $b_\gamma$ and $\tilde b_\gamma$ all satisfy for
$x\neq 0$ the equation
\[
u'=\frac{1}{\sqrt{2}}(1-u^2),
\] 
we have for all $\gamma\neq 0$
\begin{align*}
E_{\gamma}(\kappa)&=\int_0^{+\infty}\left(1-\kappa^2(x)\right)^2dx,\\
E_{\gamma}(b_\gamma)&=
             \int_{-c_\gamma}^{+\infty}\left(1-\kappa^2(x)\right)^2dx+\frac{\gamma}{2}\kappa^2(-c_\gamma),\\
E_{\gamma}(\tilde b_\gamma)&=
                    \int_{c_\gamma}^{+\infty}\left(1-\coth^2\left(\frac{x}{\sqrt{2}}\right)\right)^2dx+\frac\gamma2\coth^2\left(\frac
                   { c_\gamma}{\sqrt{2}}\right).
\end{align*}
For $x \in \R$ we set $\f(x) = \tanh(x/\sqrt 2)$. With a partial integration we compute for $\a \in \R$
\begin{eqnarray} \label{IPP-alpha}
\lefteqn{\int_{\a}^{+\infty} \left( 1 - \tanh^2\left(\frac x {\sqrt 2} \right) \right) \, dx = 2\int_{\a}^{+\infty} \f'(x)^2 \, dx}\\
\nonumber && = - \sqrt 2 \tanh\left( \frac \a {\sqrt 2} \right) \left(1-\tanh^2\left( \frac \a {\sqrt 2} \right) \right) + 2 \int_{\a}^{+\infty} \tanh^2\left(\frac x {\sqrt 2} \right)\tanh'\left(\frac x {\sqrt 2} \right) \, dx\\
\nonumber && = - \sqrt 2 \tanh\left( \frac \a {\sqrt 2} \right) \left(1-\tanh^2\left( \frac \a {\sqrt 2} \right) \right) + \frac {2 \sqrt 2} 3 - \frac {2 \sqrt 2} 3  \tanh^3\left( \frac \a {\sqrt 2} \right).
\end{eqnarray}
With $\a = 0$ we obtain 
\begin{equation} \label{energy-kappa}
E_{\gamma}(\kappa) = \frac {2\sqrt 2}3.
\end{equation}
Now let $\g \neq 0$. Using the identity 
\[
\tanh\left( \frac \a 2 \right) = \frac {\sinh (\a)} 2 \left( 1 - \tanh \left( \frac \a 2 \right)^2 \right)
\]
we obtain 
\begin{equation} \label{eq-Th}
\Theta_\gamma = \frac {\sqrt 2} {\gamma} (1- \Theta_\gamma^2), \quad \text{where }\Theta_\g = \tanh \left( - \frac {c_\gamma}{\sqrt 2} \right).
\end{equation}
Notice that since $\Theta_\gamma$ and $\gamma$ have the same sign we obtain
\begin{equation*}
\Theta_\gamma = \frac {\sign(\gamma)}{2\sqrt 2} \left( \sqrt{\gamma^2 + 8} - \abs \gamma \right).
\end{equation*}
By \eqref{IPP-alpha}, \eqref{energy-kappa} and \eqref{eq-Th} we have
\[
E_\gamma(b_\gamma) - E_\gamma(\kappa) = - \Theta_\gamma^2 \left( \frac \gamma 2 + \frac {2 \sqrt 2} 3 \Theta_\gamma \right),
\]
which proves in particular that $E_\gamma(b_\gamma) - E_\gamma(\kappa)$ and $\gamma$ have opposite signs. 
It remains to consider $E_\gamma(\tilde b_\gamma)$ when $\gamma<0$. 
Define 
\[
\tilde \Theta_\gamma=\coth\left(\frac{c_\gamma}{\sqrt{2}}\right)=-\Theta_\gamma^{-1}.
\]
By \eqref{eq-Th}, we also have
\begin{equation}\label{eq-tilde-Th}
\tilde\Theta_\gamma=\frac{\sqrt{2}}{\gamma}\left(1-\tilde\Theta_\gamma^2\right)
\end{equation}
As before, we obtain
\[
\int_{c_\gamma}^{+\infty}\left(1-\coth\left(\frac{x}{\sqrt{2}}\right)\right)^2dx=-\gamma\tilde\Theta_\gamma-\frac{2\sqrt{2}}{3}\tilde\Theta_\gamma^3+\frac{2\sqrt{2}}{3}
\]
Using \eqref{eq-tilde-Th} to linearize $\tilde \Theta_\gamma^3$, we
get
\[
E_\gamma(\tilde b_\gamma)=\frac\gamma6-\frac{2\sqrt{2}}{3}\left(\tilde\Theta_\gamma-1\right)
\]
Since $\tilde\Theta_\gamma>1$, we have $E_\gamma(\tilde b_\gamma)<0$
and therefore
\[
E_\gamma(\tilde b_\gamma)<E_\gamma(\kappa).
\]
  The alternative \eqref{comparaison-energies} follows, and Proposition \ref{prop:minimization} is proved.
\end{proof}

\section{Stability and Instability of the Black Solitons}
\label{sec:stability}

In this section we prove the orbital stability of the set of minimizers of the energy and the linear instability of the kink $\k$ when $\g > 0$.

\subsection{Stability of Black Solitons} \label{subsec:stability}

We begin with the proof of part (i) (Stability) in Theorem \ref{thm:instability}, which is a
consequence of Theorem \ref{th-pb-cauchy} and Proposition \ref{prop:minimization}.

\begin{proof}[Proof of part (i) (Stability) in Theorem \ref{thm:instability}]
  We argue by contradiction. Let $\eps>0$ and let $(u_{0,n})$ be a sequence of initial conditions in $\Ec$. For $n \in \N$ we denote by $u_n$ the solution of \eqref{eq:gp} for which $u_n(0) = u_{0,n}$. Then we assume by contradiction that
  \begin{equation*}
    \lim_{n\to+\infty}d_0(u_{0,n},\mathcal G_\gamma)=0
  \end{equation*}
  and
  \begin{equation*}
    \forall n \in \N, \exists t_n \in \R, \quad    d_0(u_n(t_n),\mathcal G_\gamma)> \eps.
  \end{equation*}
  By conservation of energy (see Theorem \ref{th-pb-cauchy}) we have
  \[
  E_\g(u_n(t_n)) = E_\g(u_{0,n}) \limt n \infty m_\g,
  \]
  and by the compactness of the minimizing sequences (see Proposition \ref{prop:minimization}) we deduce that, up to a subsequence,
  \[
  \lim_{n\to+\infty}d_0(u_n(t_n),\mathcal G_\gamma)=0.
  \]
  This gives a contradiction and finishes the proof.
\end{proof}

\subsection{Instability  of Black Solitons}

This section is devoted to the proof of part (ii) (Instability) of
Theorem \ref{thm:instability}. For this we have to prove that the
operator $L$ defined in \eqref{eq:def-L} has a negative eigenvalue.

We consider the selfadjoint operators defined on the domain $D(H_\g)$ (see \eqref{dom-Hg}) by
\[
L_-^\gamma = H_\gamma -(1-\kappa^2)\\
\quad \text{and} \quad
L_+^\gamma = H_\gamma +2-3(1-\kappa^2).
\]
These are the selfadjoint operators corresponding to the forms defined on $H^1(\R)$ by
\begin{gather*}
q_-^\gamma(u) = \nr{u'}_{L^2}^2 + \g \abs{u(0)}^2 - \int_\R(1-\kappa^2)|u|^2dx,
\\
q_+^\gamma(u) = \nr{u'}_{L^2}^2 + \g \abs{u(0)}^2  + 2 \nr{u}_{L^2}^2-3\int_\R(1-\kappa^2)|u|^2dx.
\end{gather*}

Separating the real and imaginary parts of $\eta$ in \eqref{eq-eta} gives the system
\[
\partial_t\binom{\Re(\eta)}{\Im(\eta)}+\mathcal L\binom{\Re(\eta)}{\Im(\eta)}+\mathcal N(\eta)=0,
\]
where 
\begin{equation*}
  \mathcal{L}=\begin{pmatrix}0 & -L_-^\gamma\\L^\gamma_+&0\end{pmatrix}
  \quad \text{and} \quad
  \mathcal N(\eta) = \binom{\Re(N(\eta))}{\Im(N(\eta))}.
\end{equation*}
Notice that a real eigenvalue of $\mathcal L$ is also an eigenvalue of
$L$, so the proof of part (ii) (Instability) in Theorem
\ref{thm:instability} reduces to proving that $\mathcal L$ has a real negative eigenvalue. We start by analyzing $L_+^\gamma$ and $L_-^\gamma$.

\begin{proposition}[Spectral properties of  $L_\pm^\gamma$] \label{prop-spectre-L}
  Let $\gamma \in \R$.
  \begin{enumerate}[(i)]
  \item The essential spectra of $L_-^\gamma$ and $L_+^\gamma$ are $\s_{\mathrm{ess}}(L_-^\gamma) = [0,+\infty)$ and $\s_{\mathrm{ess}}(L_+^\gamma) = [2,+\infty)$.
  \item The operator $L_-^\gamma$ has a trivial kernel and at least one negative eigenvalue.
  \item If $\gamma<0$ then $L_+^\gamma$ has a trivial kernel and a unique negative eigenvalue. If $\gamma = 0$, then 0 is the first eigenvalue of $L_+^0$. If $\gamma > 0$ then $L_+^\gamma$ has no eigenvalue in $(-\infty,0]$.
  \end{enumerate}
\end{proposition}

\begin{proof}
  We know that the essential spectrum of $H_\gamma$ is $[0,+\infty)$ (see Theorem 3.1.4 in \cite{AlGeHoHo88}). This implies in particular that the essential spectrum of $H_\gamma + 2$ is $[2,+\infty)$. Since 
  \[
  1-\kappa^2 (x) = \sech^2\left(\frac{x}{\sqrt{2}}\right) \limt {\abs x} {+\infty} 0,
  \]
  the first statement follows from Weyl Theorem.

  The forms $q_-^\gamma$ and $q_+^\gamma$ are analytic with respect to $\gamma$, so $L_-^\gamma$ and $L_+^\gamma$ define analytic families of operators of type B in the sense of Kato (see \S VII.4 in \cite{kato}). In particular, if $I$ is an open interval of $\R$ and $a,b \in \R$ are in the resolvent set of $L_+^\g$ for all $\g \in I$, the the spectral projection $\Pi_{a,b}^\gamma$ of $L_+^\gamma$ on $(a,b)$ is an analytic family of orthogonal projections, and the spectrum of the restriction of $L_+^\gamma$ on $\Pi_{a,b}^\gamma L^2(\R)$ is $\s(L_+^\gamma) \cap (a,b)$.

  For $\gamma=0$, we can check that the spectrum of $L_+^0$ is included in $[0,+\infty)$ and that 0 is a simple eigenvalue of $L_+^0$. Indeed, by differentiating with respect to $x$ the equation \eqref{SGPdelta} satisfied by $\kappa$, we see that $\kappa'$ belongs to the kernel of $L_+^0$. Since it takes positive values on $\mathbb R$, this implies that $0$ is simple and is the first eigenvalue of $L_+^0$. Similarly, we check by direct computation that $u_-^0 : x \mapsto \sech\left({x}/{\sqrt{2}}\right)$ takes positive values and is an eigenfunction for $L_-^0$ corresponding to the eigenvalue $-\frac12$.

  By analyticity of the spectrum of $L_+^\gamma$, there exist $\n > 0$ and two analytic functions $\l : (-\n,\n) \to \R$ and $u : (-\n,\n) \to L^2(\R)$ such that $\l(0) = 0$, $u(0) = \kappa'$ and, for all $\gamma \in (-\n,\n)$, $\l(\gamma)$ is the first eigenvalue of $L_+^\gamma$, it is simple, and $u(\gamma)$ is a corresponding eigenfunction. On the one hand we have
  \[
  \dual{L_+^\gamma u(\gamma)}{u(\gamma)}=\lambda(\gamma)\norm{u(\gamma)}_{L^2}^2= \gamma \lambda'(0) \norm{\kappa'}_{L^2}^2+O(\gamma^2).
  \]
  On the other hand
  \[
  \dual{L_+^\gamma u(\gamma)}{u(\gamma)}
  =\dual{L_+^0 u(\gamma)}{u(\gamma)}+\gamma |u(\gamma)(0)|^2=
  \gamma |\kappa'(0)|^2+O(\gamma^2),
  \]
  so
  \[
  \lambda'(0) = \frac{ |\kappa'(0)|^2}{\norm{\kappa'}_{L^2}^2} > 0.
  \]
  Thus $\lambda(\gamma)$ has the same sign as $\gamma$ for $|\gamma|$ small enough.

Let $\Gamma > 0$. Assume that $\g \in [-\G,0]$, $u \in D(\Hg)$ and $\l \in(-\infty,0]$ are such that $\nr{u}_{L^2} = 1$ and $L_+^\g u = \l u$. Since $\abs{u(0)}^2 \leq 2 \nr{u}\nr{u'}$ we have
\[
\l = q_+^\g(u) \geq  \nr{u'}_{L^2}^2 - \abs \g \abs{u(0)}^2 - 1 \geq \nr{u'}_{L^2}^2 - 2 \abs \g \nr{u'}_{L^2} - 1.
\]
This proves that there exists $C > 0$ such that for $\g \in [-\G,0]$ the operator $L_+^\g$ has no eigenvalue in $(-\infty,-C]$.

If we prove that the kernel of $L_+^\gamma$ is trivial for all $\gamma < 0$, this will imply that $\l$ extends to an analytic function on $(-\G,0)$ which gives the unique eigenvalue of $L_+^\g$ in $(-C , 0)$ and hence in $(-\infty,0)$. Indeed the projection $\Pi_{-C,0}^\gamma$ is analytic for $\g \in (-\Gamma,0)$. Since it is of rank 1 for $\abs \g$ small enough, it is of rank 1 for all $\g \in (-\Gamma,0)$, which means that $\s(L_+^\gamma) \cap (-C,0) = \s(L_+^\gamma) \cap (-\infty,0)$ consists on a simple eigenvalue. Since the choice of $\Gamma$ is arbitrary, this will prove that for any $\gamma < 0$ the operator $L_+^\g$ has a unique negative eigenvalue.

  So let $\gamma \in \R \setminus \{0\}$ and $v\in\ker(L_+^\gamma)$. Then $v$ satisfies 
  \[
  -v'' + 2v - 3(1-\kappa^2)v = 0 \quad \text{on }(0,+\infty).
  \]
  Since $\kappa'$ solves the same equation, there exists $C \in \R$ such that 
  \[
  v \kappa'' - v' \kappa' = C \quad \text{on } (0,+\infty).
  \]
  Then, since $v \in L^2(0,+\infty)$, this implies that there exists $\alpha \in \R$ such that $v = \alpha \kappa'$ on $(0,+\infty)$. Similarly, there exists $\beta \in\R$ such that $ v = \b \kappa'$ on $(-\infty,0)$. Since $v$ is continuous and $\kappa' (0) \neq 0$, we have $\a = \b$. And finally, the jump condition 
  \begin{equation*}
    v'(0^+) - v'(0^-) = \gamma v(0)
  \end{equation*}
  implies that $\a=\b=0$ and hence $v = 0$. This proves that $\ker(L_+^\gamma)=\{0\}$ and concludes the proof of the third statement.

  Now we check that we also have $\ker(L_-^\gamma)=\{0\}$ for any $\gamma\in\R$. Indeed, if $v$ satisfies the equation 
  \[
  -v'' - (1-\kappa^2)v = 0 \quad \text{on }(0,+\infty),
  \]
  then it is not hard to find out that there exist $C_1,C_2\in\R$ such that
  \[
  v(x)=C_1 \big( \sqrt{2}+x\kappa(x) \big)+C_2\kappa(x).
  \]
  Since $v \in L^2(\R)$, we necessarily have $(C_1,C_2) = (0,0)$, and hence $\ker(L_-^\gamma)=\{0\}$.

  It remains to show that $L_-^\gamma$ has at least one negative eigenvalue. For this we prove that there exists $v\in H^1(\R)$ such that
  $
  q_-^\g(v) < 0.
  $
  For $\gamma \leq 0$ we can take the eigenfunction $v= u_-^0$ of $L_-^0$. For $\gamma>0$, we need a more refined construction. Let $\h \in C_0^\infty(\R,[0,1])$ be equal to 1 on a neighborhood of 0. For $r \geq 1$ and $x \in \R$ we set $\chi_r(x) = \h(x/r)$ and
  \[
  v_r(x) = \k(\abs x) \h_r(x) + \frac {\sqrt 2} {\g} u_-^0(x) =  \k(\abs x) \h_r(x) + \frac {\sqrt 2} {\g} \sech \left(\frac{x}{\sqrt{2}}\right).
  \]
  We first remark that 
  \[
  v_r'(0^+) - v_r'(0^-) = 2 v_r'(0^+) = \g v_r(0),
  \]
  and then that $v_r \in D(H_\gamma)$ for all $r \geq 1$. Therefore
  \begin{align*}
    q_\gamma^-(v_r)
    & = \innp{-v_r''-(1-\k^2)v_r}{v_r}\\
& = 2 \int_0^{+\infty} \Big(\big(-\k'' - (1-\k^2) \k \big) \h_r -2 \k' \h_r' - \k \h_r'' \Big) \overline{v_r}  + \frac {\sqrt 2} \g \innp{L_-^0 u_-^0}{v_r}\\
    & = - 2 \int_0^{+\infty} (2 \k' \h_r' + \k \h_r'') \overline{v_r} - \frac 1 \g \innp{u_-^0}{v_r}.
  \end{align*}
  By the dominated convergence theorem we have
  \[
  \limsup_{r\to +\infty} q_\gamma^-(v_r) = \limsup_{r\to +\infty} \left( - \frac 1 \g \innp{u_-^0}{v_r} \right) < 0,
  \]
  so there exists $r \geq 1$ such that $q_\gamma^-(v_r) < 0$. This concludes the proof of the proposition.
\end{proof}

\begin{remark}
The number of negatives eigenvalue for
  $L_-^\gamma$ and  $L_+^\gamma$  gives no hint toward stability/instability,
unlike what
  was happening for the localized standing waves studied in \cite{LeFuFiKsSi08} where it was possible to appeal to Grillakis-Shatah-Strauss Theory.
\end{remark}

Now we can prove part (ii) (Instability) of Theorem \ref{thm:instability}.

\begin{proof}[Proof of part (ii) (Instability) of Theorem \ref{thm:instability}]
  Let $\g > 0$. We have to show that $\mathcal L$ has a real negative eigenvalue. Since the operator $L_+^\gamma$ is positive, we can set
  \[
  \Lambda=(L_+^\gamma)^{\frac 12}L_-^\gamma (L_+^\gamma)^{\frac12}.
  \]
  This defines a selfadjoint operator on the domain
  \[
  D(\Lambda)=\left\{ u\in
    D \big( (L_+^\gamma)^{\frac12} \big)\, : \,(L_+^\gamma)^{\frac12}u\in D(L_-^\gamma) \, \text{ and } \,L_-^\gamma
    (L_+^\gamma)^{\frac12}u\in D \big((L_+^\gamma)^{\frac12} \big)\right\}.
  \]

  Assume that $w \in D(\L) \setminus \{0\}$ and $\l \in \R \setminus \{0\}$ are such that $\L w = - \l^2 w$. Then we set $u = (L_+^\gamma)^{-\frac12} w$ and $v = \frac{1}{\lambda} (L_+^\gamma)^{\frac 12} w$ (notice that $(L_+^\gamma)^{-\frac12}$ is a bounded operator on $L^2(\R)$ since the spectrum of $L_+^\g$ is included in $[\nu,+\infty)$ for some $\nu > 0$). By construction we have $u \in D \big( (L_+^\gamma)^{\frac12} \big)$ and $(L_+^\gamma)^{\frac 12}u \in D \big((L_+^\gamma)^{\frac12} \big)$, so $u \in D(L_+^\gamma)$. We also have $v \in D(L_-^\gamma)$. Moreover, 
  \[
  -L_-^\gamma v = \lambda u
  \quad \text{and} \quad 
  L_+^\gamma u = \lambda v,
  \] 
  so $\lambda$ is an eigenvalue of $\mathcal L$. Thus it remains to prove that $\Lambda$ has a negative eigenvalue. For this we prove that its essential spectrum is non-negative while its full spectrum has a negative part.

  We denote by $\Pi_-^\g$ and $\Pi_+^\g$ the spectral projections of $L_-^\g$ on $(-\infty,0)$ and $[0,+\infty)$, respectively. Then we set $\L_\pm = (L_+^\gamma)^{\frac 12}\Pi_\pm^\gamma (L_+^\gamma)^{\frac12}$. By Proposition \ref{prop-spectre-L}, $\Pi_-^\g$ is of finite rank, so it is a compact operator from $L^2(\R)$ to $D(L_-^\g) = D(L_+^\g)$. This implies that $\L_-$ is a relatively compact perturbation of $\L$. Thus, by the Weyl Theorem, $\L$ and $\L_+$ have the same essential spectrum. But $\L_+$ is a non-negative operator, so $\s_{\mathrm{ess}}(\L) \subset [0,+\infty)$.

  Now let $\xi \in D(L_-^\gamma)$ be an eigenfunction corresponding to a negative eigenvalue $\lambda$ of $L_-^\gamma$ and $\eta =(L_+^\gamma)^{-\frac12} \xi$. Then $\eta \in D \big( (L_+^\g)^{\frac 12} \big)$ and $(L_+^\gamma)^{\frac12}\eta = \xi \in D(L_-^\gamma)$. Moreover, $L_-^\gamma (L_+^\gamma)^{\frac12}\eta = \lambda \xi \in D(L_-^\gamma) = D(L_+^\gamma)\subset D((L_+^\gamma)^{\frac12})$. Therefore, $\eta\in D(\Lambda)$ and  
  \[
  \dual{L\eta}{\eta}=\dual{L_-^\gamma\xi}{\xi}<0 .
  \]
  This implies that the selfadjoint operator $\Lambda$ has a negative eigenvalue, which concludes the proof of the linear instability of $\kappa$.
\end{proof}

\bibliographystyle{abbrv}
\bibliography{ianni-lecoz-royer}


\end{document}